\documentclass[11pt,reqno]{amsart}
\usepackage[colorlinks=true,
pdfstartview=FitV, linkcolor=cyan, citecolor=magenta,
urlcolor=blue]{hyperref}
\usepackage{amsmath,amsfonts,latexsym,amssymb}
\usepackage{mathrsfs}
\usepackage[latin1]{inputenc}
\usepackage[T1]{fontenc}
\usepackage{ae,aecompl}
\usepackage{braket}
\usepackage{comment}
\usepackage{color}
\usepackage{graphicx}
\usepackage{subfig}
\newtheorem{theorem}{Theorem}[section]
\newtheorem{lemma}[theorem]{Lemma}
\newtheorem{proposition}[theorem]{Proposition}
\newtheorem{corollary}[theorem]{Corollary}

\renewcommand{\leq}{\leqslant}

\renewcommand{\le}{\leqslant}
\renewcommand{\ge}{\geqslant}

\long\def\@savemarbox#1#2{\global\setbox#1\vtop{\hsize\marginparwidth 
  \@parboxrestore\tiny\raggedright #2}}
\renewcommand{\d}{{\rm d}}

\newcommand{\psln}{\mathsf{PSL}_d(\mathbb R)}

\newcommand{\bgrf}{\partial_\infty\pi_1(S)}

\newcommand{\Hn}{\mathcal H_d(S)}

\newcommand{\TT}{\mathsf{T}}
\newcommand{\II}{{\bf I}}
\newcommand{\JJ}{{\bf J}}
\newcommand{\PP}{{\bf P}}
\newcommand{\s}{{\bf{s}}}
\newcommand{\p}{{\bf{p}}}
\newcommand{\q}{{\bf{q}}}

\newcommand{\Real}{\mathbb R}

\newcommand{\SL}{\mathsf{SL}(d,\Real)}

\newcommand{\ms}{\mathsf}

\newcommand{\clase}{\operatorname{C}}

\newcommand{\B}{\mathbb{B}}
\renewcommand{\P}{\mathbb{P}}

\newcommand{\g}{\gamma}

\renewcommand{\b}{{\ms{b}}}

\DeclareMathOperator{\vol}{vol}

\DeclareMathOperator{\Tr}{Tr}
\DeclareMathOperator{\Hess}{Hess}
\DeclareMathOperator{\Hol}{Hol}

\renewcommand{\sf}[1]{{\mathsf{#1}}}

\newcommand{\tR}{_{t\in\mathbb R}}

\renewcommand{\hom}{\operatorname{Hom}}
\def\eproof{$\Box$ \medskip}

\title[Simple roots and Hitchin representations]{Simple root flows for Hitchin representations}
\author[Bridgeman]{Martin Bridgeman}
\address{Boston College, Chestnut Hill, MA 02467 USA}
\author[Canary]{Richard Canary}
\address{University of Michigan, Ann Arbor, MI 41809 USA}
\author[Labourie]{Fran\c cois Labourie}
\address{ Univ. Nice Sophia-Antipolis, Laboratoire Jean Dieudonn\'e,  UMR 7351, Nice F-06000 FRANCE}
\author[Sambarino]{Andres Sambarino}
\address{Universit\'e Pierre et Marie Curie (Paris VI),
4 place Jussieu, Paris 75005 France}
\thanks{Bridgeman was partially suppported by grants DMS-1500545  and DMS-1564410 and
Canary was partially supported by  grants DMS-1306992 and DMS-1564362, from the National Science Foundation.
Labourie and Sambarino were partially supported by the European Research Council under the {\em European Community}'s seventh Framework Programme (FP7/2007-2013)/ERC {\em grant agreement} ${\rm n}^{\tiny o}$ FP7-246918 and by ANR project DynG\'eo ANR-16-CE40-0025.}

\begin{document}

\maketitle
\begin{center}
{\em Dedicated to a great mathematician on the occasion \\
of his sixtieth birthday: our friend Bill Goldman.}
\end{center}

\bigskip

\section{Introduction}

Anosov representations (\cite{labourie-anosov,guichard-wienhard}) from a hyperbolic group to a semi-simple Lie group are 
characterized by their dynamical nature. In the context of projective Anosov representations, 
 we \cite{BCLS} previously associated a metric Anosov flow to such a  representation and showed that the
(thermo)dynamical properties of this flow yield in turn new structures on the deformation space of these representations: 
entropy functions, (pressure) intersections and a pressure metric.

In this paper, we focus on Hitchin representations of a surface group into $\psln$. 
We associate a wealth of flows to a Hitchin representation, and hence geodesic currents, entropies, 
pressure forms {\it etc.}, depending essentially on an element in the Weyl chamber.

Let us be more specific.  If $\sf E$ is a real vector space of dimension $d$ and $S$ is a closed surface,
a representation $\rho:\pi_1(S)\to \sf{PSL}(E)$ is {\em $d$-Fuchsian} if it is
the composition of  a Fuchsian representation into $\sf{PSL}(2,\mathbb R)$ and
an irreducible representation of $\sf{PSL}(2,\mathbb R)$ into $\sf{PSL}(E)$.  A representation
\hbox{$\rho:\pi_1(S)\to\sf{PSL}(\sf{E})$} is a {\em Hitchin representation} if it may be continuously deformed to a $d$-Fuchsian representation.
Hitchin \cite{hitchin} showed that the {\em Hitchin component} $\Hn$ of ($\sf{PGL}(E)$-conjugacy classes of) Hitchin representations into
$\sf{PSL}(E)$ is an analytic manifold diffeomorphic to $\mathbb R^{(d^2-1)|\chi(S)|}$. 
Labourie \cite{labourie-anosov} showed that a Hitchin representation is a discrete, faithful quasi-isometric embedding and
that the image of every non-trivial element $\gamma$ is diagonalizable over $\mathbb R$ with eigenvalues of distinct modulus:
$$\lambda_1(\rho(\gamma))>\lambda_2(\rho(\gamma))>\cdots>\lambda_d(\rho(\gamma))>0.
$$
Moreover, there are H\"older-continuous, $\rho$-equivariant {\em limit curves}  \hbox{$\xi_\rho:\bgrf\to {\bf P}(\sf{E})$} and 
\hbox{$\xi_\rho^*:\bgrf\to {\bf P}(\sf{E}^*)$} whose images are 
$\clase^{1+\alpha}$-submanifolds. This last feature is very specific to Hitchin representations -- see subsection \ref{subsection.Hitchinreps} (Theorem \ref{hyperconvexity})  and Guichard \cite{Guichard:2008} for details.

Let $\mathcal G(S)=\bgrf^2\setminus\Delta$ be the space of distinct points in the Gromov boundary $\partial_\infty\pi_1(S)$
of $\pi_1(S)$.  
We say that a \emph{flow over $\mathcal G(S)$} is an $\Real$-principal  bundle ${\sf L}$ over $\mathcal G(S)$ equipped with a 
properly discontinuous and co-compact action of $\pi_1(S)$ by bundle automorphisms. 
The $\Real$-action on the quotient space ${\sf U}_L={\sf L}/\pi_1(S)$ is a flow, which justifies the terminology. 
Given a geodesic current $\omega$, {\it {\it i.e. }} a $\pi_1(S)$-invariant locally finite measure 
on $\mathcal G(S)$, we define a pairing
$$ \braket{\omega\mid {\sf L}} :=\int_{{\sf U}_{\sf L}} \omega\otimes \d t\ $$ 
where $\d t$ is the element of arc length given by the $\Real$ action. 

We focus on the {\em simple root flows}  associated to a Hitchin representation $\rho$ (see Section \ref{simple root flows}).
For each $i\in\{1,\ldots,d-1\}$ there is a flow ${\sf L}^{\alpha_i}_\rho$ over $\mathcal G(S)$ such that  if $\delta_\gamma$ is the geodesic current 
with Dirac measure one on every (oriented) axis of an element conjugate to $\gamma$, then 
 $$\braket{\delta_\gamma\mid {\sf L}^{\alpha_i}_\rho}= L_{\alpha_i}(\rho(\gamma)) :=\log\left(\frac{\lambda_i(\rho(\gamma))}{\lambda_{i+1}(\rho(\gamma))}\right)\ .$$ 
Equivalently, if we let $\sf{U}_{\alpha_i}(\rho)$ be the quotient flow with associated 
element of arc length $\d s_\rho^{\alpha_i}$, then the period of $\sf{U}_{\alpha_i}(\rho)$ associated to
$\gamma\in\pi_1(S)$ is given by $L_{\alpha_i}(\rho(\gamma))$, the {\em $L_{\alpha_i}$-length function}.

We also consider the {\em Hilbert flow} ${\sf L}^{\sf H}(\rho)$ associated to $\rho$ which is determined,
up to H\"older conjugacy, by 
$$\braket{\delta_\gamma\mid {\sf L}^{\sf{H}}_\rho}=L_H(\rho(\gamma)) :=\log\left(\frac{\lambda_1(\rho(\gamma))}{\lambda_{d}(\rho(\gamma))}\right)\ ,$$ 
for any non-trivial $\g\in\pi_1(S)$.

Potrie and Sambarino show that the entropy of simple root flows is constant and characterize Fuchsian representations
in terms of the entropy of the Hilbert flow.

\begin{theorem}{\rm(Potrie--Sambarino \cite{potrie-sambarino})}
The topological entropy of a simple root flow is $1$ for all Hitchin representations. 
Moreover, a Hitchin representation $\rho\in\Hn$ is $d$-Fuchsian if and only if the topological
entropy of the Hilbert flow is $\frac{2}{d-1}$.
\end{theorem}

One of the main constructions of our paper is to single out, amongst all geodesic currents associated to a Hitchin representation, 
a specific asymmetric  current called the {\em Liouville current}  $\omega_\rho$. 
This Liouville current was introduced in \cite{labourie-cross}  and  characterized 
by the cross ratio ${\rm b}_\rho$ of $\rho$  as discussed in  Section \ref{sec:cr}. 
If $(t,x,y,z)$ are four points in cyclic order in $\bgrf$, then
$$
\omega_\rho\left([t,x]\times [y,z]\right)=\frac{1}{2}\log \left(\frac{\langle u\mid \Phi\rangle \langle v\mid \Psi\rangle} 
{\langle u\mid \Psi\rangle \langle v\mid \Phi\rangle}\right).
$$  
where $u$, $v$, $\Phi$ and $\Psi$ are non zero elements in $\xi_\rho(t)$, $\xi_\rho(x)$, $\xi_\rho^*(y)$ and $\xi_\rho^*(z)$ respectively.

As a consequence of Labourie's work on cross ratios for Hitchin representations \cite{labourie-cross},
this gives an embedding of the space of all Hitchin representations into the space of geodesic currents.

\begin{theorem}
\label{Liouville current determines}
If $\rho$ and $\sigma$ are two Hitchin representations --of possibly different dimensions -- with the same Liouville current, 
then $\rho=\sigma$.	
\end{theorem}

The Liouville current enjoys the following properties.

\begin{theorem}
\label{propertiesLiouville} 
If $\rho$ is a Hitchin representation, then
\begin{enumerate}
\item  
The current $\omega_\rho$ is the unique current -- up to scalar multiplication --  in the 
class of the Lebesgue measure for the $\clase^1$ structure on $\mathcal G(S)$ associated to the embedding $(\xi,\xi^*)$.
\item 
The measure $\omega_\rho\otimes \d s^{\alpha_1}_\rho$ is  -- up to scalar multiplication-- the unique measure maximizing entropy 
for the flow $\sf{U}_{\alpha_1}(\rho)$.
\item 
If $\mu$ is a geodesic current, then 	
$$i(\mu,\omega_\rho)=\braket{\mu\mid {\sf L}^{\sf H}_\rho}.$$
\end{enumerate}
\end{theorem}

Our Liouville current is closely related to the symmetric Liouville currents defined by Bonahon \cite{bonahon},
when $d=2$, and Martone-Zhang \cite{martone-zhang}. In fact, one may view their Liouville currents as
symmetrizations of our Liouville current.

We define the {\em Liouville volume} of a representation, by
$$
\vol_{\sf L}(\rho)=i(\omega_\rho,\omega_\rho),
$$
and establish the following volume rigidity result, which is motivated by work of Croke and Dairbekov 
\cite{Croke-Dairbekov}.

\begin{theorem}
\label{thurston metric version}
If $\rho,\eta\in\Hn,$  then
$$\left(\inf_{\gamma\in\pi_1(S)\setminus\{1\}}\frac{\braket{\delta_\gamma\mid {\sf L}^{\sf{H}}_\rho}}{\braket{\delta_\gamma\mid {\sf L}^{\sf{H}}_\eta}}\right)^2\le \frac{\vol_{\sf{L}}(\rho)}{\vol_{\sf{L}}(\eta)} \le \left(\sup_{\gamma\in\pi_1(S)\setminus\{1\}} \frac{\braket{\delta_\gamma\mid {\sf L}^{\sf{H}}_\rho}}{\braket{\delta_\gamma\mid {\sf L}^{\sf{H}}_\eta}}\right)^2$$
and equality holds in either inequality if and only if  either $\rho=\eta$ or $\rho=\eta^*$ where
$\eta^*$ is the contragredient of $\eta$.
\end{theorem}

When $d=3$, we apply work of  Tholozan \cite[Thm. 3]{tholozan} to obtain a simpler volume rigidity result.

\begin{corollary}
If $\rho\in\mathcal H_3(S)$, then
$$\vol_{\sf{L}}(\rho)\ge 4\pi^2|\chi(S)|.$$
Moreover, equality holds if and only if $\rho$ is $3$-Fuchsian.
\end{corollary}

We return to the themes explored in \cite{BCLS}, by constructing a new, hopefully more tractable, pressure metric on
a Hitchin component.
If $\rho,\eta\in\Hn$,
we define their {\em Liouville pressure intersection} to be
$$
\II_{\alpha_1}(\rho,\eta):=
\frac{1}{\braket{\omega_\rho\mid {\sf L}^{\alpha_1}_\rho}}\braket{\omega_\rho\mid {\sf L}^{\alpha_1}_\eta}.
$$ 
If $\rho\in\Hn$, 
we define a function $(\II_{\alpha_1})_\rho:\Hn\to \Real$  by letting $(\II_{\alpha_1})_\rho(\eta)=\II_{\alpha_1}(\rho,\eta)$. 

Using the thermodynamic formalism developed by Bowen \cite{bowen-book}, Ruelle \cite{ruelle-book} and 
Parry-Pollicott \cite{parry-pollicott} we show that  $(\II_{\alpha_1})_\rho$ has a minimum at $\rho$, 
and its Hessian $\PP_{\alpha_1}$ at $\rho$ is positive semi-definite. We call $\PP_{\alpha_1}$ the 
{\em Liouville pressure quadratic form}. This construction is motivated by Thurston's version of the
Weil-Petersson metric on Teichm\"uller space (see Wolpert \cite{wolpert}) as
re-interpreted by Bonahon \cite{bonahon}, McMullen \cite{mcmullen-pressure} and 
Bridgeman \cite{bridgeman-wp}.

We show that $\PP_{\alpha_1}$ is non-degenerate, hence gives rise to a Riemannian metric, and apply
work of Wolpert \cite{wolpert} to see that it restricts to a multiple of the Weil-Petersson metric
on the Fuchsian locus.

\begin{theorem}\label{main} 
The Liouville pressure quadratic form $\PP_{\alpha_1}$ 
is a mapping class group invariant, analytic Riemannian metric on $\Hn$,
that restricts to a scalar multiple of the the Weil-Petersson metric on the Fuchsian locus.
\end{theorem}

The main tool in the proof of the non-degeneracy of $\PP_{\alpha_1}$ is that the $L_{\alpha_1}$-length functions 
of elements of $\pi_1(S)$ generate
the cotangent space of the Hitchin component. More precisely, 
if $\gamma\in\pi_1(S)$, let $L_{\alpha_1}^\gamma:\Hn\to \mathbb R$ be given by 
$$L_{\alpha_1}^\gamma(\rho)=L_{\alpha_1}(\rho(\gamma))=\braket{\delta_\gamma\mid {\sf L}^{\alpha_1}_\rho}.$$

 \begin{theorem}
 \label{cotangent intro}
If $\rho\in\Hn$, then the set
$$\{ {\rm D}_\rho L_{\alpha_1}^\gamma\}_{\gamma\in\pi_1(S)}$$ 
generates, as a vector space, the cotangent space $\TT_\rho^*\Hn.$
 \end{theorem}

We can also give an interpretation of $\II_{\alpha_1}$ in terms more reminiscent  of the construction in \cite{BCLS}.
This interpretation
generalizes to give pressure quadratic forms associated to other simple roots.
If  $T>0$ and $i\in\{1,\ldots,d-1\}$, let
$$R_{\alpha_i}(\rho,T)=\{[\gamma]\in [\pi_1(S)]\setminus\{[1]\}\ |\ L_{\alpha_i}(\rho(\gamma))\le T\}.$$
We then define an associated pressure intersection  
$$\II_{\alpha_i}(\rho,\eta)=\lim_{T\to\infty}\frac1{\# R_{\alpha_i}(\rho,T)}\sum_{\g\in R_{\alpha_i}(\rho,T)}\frac{ L_{\alpha_i}(\eta(\g))}{ L_{\alpha_i}(\rho(\g))}.$$ 
The associated function $(\II_{\alpha_i})_\rho$ has a minimum at $\rho$, and we  again obtain, by
considering the Hessian of $(\II_{\alpha_1})_\rho$, a positive semi-definite quadratic  pressure form $\PP_{\alpha_i}$.
It is natural to ask when $\PP_{\alpha_i}$ is non-degenerate. 
In a final section, we observe that $\PP_{\alpha_n}$ is degenerate on $\mathcal H_{2n}(S)$ at any Hitchin
representation with image (conjugate into) $\mathsf{PSp}(2n)$, see Proposition \ref{Pn degenerate on H2n}.

We recall that our original pressure metric from \cite{BCLS} was obtained as the Hessian of a renormalized pressure
intersection 
$$\JJ(\rho,\eta)=\frac{h(\rho)}{h(\eta)} \lim_{T\to\infty}\frac1{\# R_1(\rho,T)}\sum_{\g\in R_1(\rho,T)}\frac{ L_1(\eta(\g))}{ L_1(\rho(\g))}$$ 
where $L_1(\rho(\gamma))=\log\lambda_1(\rho(\gamma))$, 
$R_1(\rho,T)=\{[\gamma]\in [\pi_1(S)]\setminus\{[1]\}\ |\ L_1(\rho(\gamma))\le T\}$ and 
the spectral radius entropy $h(\rho)$ is the exponential growth rate of $R_1(\rho,T)$.

There are two main advantages of the Liouville pressure metric with respect to the pressure metric defined in \cite{BCLS}. 
First, due to  work of  Potrie and Sambarino \cite{potrie-sambarino},
we do not have to renormalize the pressure intersection by an entropy. Second, the Bowen--Margulis measure associated
to the first simple root is directly related to the cross ratio of the representation. 
We hope that these two facts will make the Liouville  pressure metric more accessible to computation.
It follows from work of Zhang \cite{zhang2} and Theorem \ref{thurston metric version} that the Liouville
volume is non-constant on $\Hn$ when $d\ge 3$, so one cannot
directly use the Hessian of intersection to construct a metric, as Bonahon \cite[Thm. 19]{bonahon} does  to
reconstruct the Weil-Petersson metric when $d=2$.

\medskip\noindent
{\bf Acknowledgements:} The authors would like to thank Harrison Bray, Francois Ledrappier, Ralf Spatzier and Tengren Zhang for helpful
conversations.

\section{Dynamical background}

In Sections \ref{therm defs} and \ref{expansion} we
recall the thermodynamic formalism of Bowen and Ruelle (\cite{bowen-book, bowen-ruelle, ruelle-book}),
which was further developed by Parry and Pollicott  \cite{parry-pollicott}. We then
discuss geodesic currents (in Section \ref{geodesic currents}) and describe the relationship between contracting
line bundles and flows (in Section \ref{contracting line bundles}).

\subsection{Basic definitions}
\label{therm defs}

Let $X$ be a compact metric space  and \hbox{$\phi=\{\phi_t:X \rightarrow X\}\tR$} be a topologically 
transitive, metric Anosov flow on $X$.
(Metric Anosov flows were first defined by Pollicott \cite{pollicott-smale}
who called them Smale flows.)
Let $O_\phi$ be the collection of periodic orbits of the flow $\phi$ and
and define
$$R_\phi(T) = \{ a\in O_\phi\ | \ \ell(a) \le T\}$$
where $\ell(a)$ is the period of $a$.
The {\em topological entropy} of the flow $\phi$ is given  by 
$$h(\phi) = \lim_{T\rightarrow \infty} \frac{\log\#R_\phi(T)}{T}.$$

If $\alpha>0$, let $\Hol^\alpha(X,\Real)$ be the space of $\alpha$-H\"older continuous functions on $X$. 
If \hbox{$f \in \Hol^\alpha(X,\Real)$}, let
$$\ell_f(a) = \int_X f {\rm d}\hat\delta_a$$
where $\hat\delta_a$ is a $\phi$-invariant measure supported on $a$ with total mass $\ell(a)$.
Let 
$$R_\phi(f,T)= \{ a\in O_\phi\ | \ \ell_f(a) \le T\}$$
and define
$$h_\phi(f) = \lim_{T\rightarrow \infty} \frac{\log\#R_\phi(f,T)}{T}.$$

If $f$ is positive, we obtain a new flow $\phi^f$ on $X$ by reparametrizing $\phi$ by $f$. Concretely, $\phi^f$ is determined by the formula
$$\phi^f_{k_f(x,t)}(x) = \phi_t^f(x)$$
where $k_f(x,t) = \int_0^t f(\phi_sx)\d s$ 
for all $x\in X$ and $t\in\mathbb R$. Notice that if ${\rm d}s$ is an element of arc length for the flow
lines of $\phi$, then $f{\rm d}s$ is an element of arc length for the flow lines of $\phi^f$.

The flow $\phi^f$
is H\"older orbit equivalent to $\phi$ and if $a\in O_\phi=O_{\phi^f}$, then $\ell_f(a)$ is the period
of $a$ in the flow $\phi^f$. In this case, $h_\phi(f)$ is the topological entropy $h(\phi^f)$ of the flow $\phi^f.$

We will say that $f,g\in \Hol^\alpha(X,\Real)$ are {\em Liv\v sic cohomologuous} if there exists $U:X\to\Real$ 
such that for all $x\in X$ one has 
$$f(x)-g(x)=\left.\frac{\partial}{\partial t}\right|_{t=0}U(\phi_t x).$$
Recall that $f$ and $g$ are Liv\v sic cohomologous if and only if $\ell_f(a)=\ell_g(a)$ for all $a\in O_\phi$.
Moreover, if  $f$ and $g$ are positive, then $\phi^f$ and $\phi^g$ are H\"older conjugate if and only if 
$f$ and $g$ are Liv\v sic cohomologuous (see Liv\v sic \cite{livsic}).

If $\mathcal M_\phi$ is the space of $\phi$-invariant probability measures on $X$ and $m \in \mathcal M_\phi$, 
let $h(\phi,m)$ be the metric entropy of $m$.  Then, for $f \in \Hol^\alpha(X,\Real)$, the {\em topological pressure} is
$$\PP_\phi(f)=\sup_{m\in \mathcal M_\phi}\left\{h(\phi,m)+\int_X f {\rm d}m\right\}.$$
A measure that attains this supremum is called an {\em equilibrium state} for $f$ and an equilibrium state for the zero function is called a {\em measure of maximal entropy}.  

If $f\in\Hol^\alpha(X,\Real)$ is positive, Bowen
\cite[Thm. 5.11]{bowen-periodic} (see also Pollicott \cite[Thm. 9]{pollicott-smale}) showed that the measure of
maximal entropy for $\phi^f$ is given by the {\em Bowen-Margulis measure} for $\phi^f$
$$\lim_{T\to\infty} \frac{1}{\#R_\phi(f,T)}\sum_{a\in O_X}  \frac{\hat \delta_a}{\ell_f(a)}$$
where $\hat\delta_a$ is the product of Dirac measure on the orbit $a$ and the element of arc length on $a$ in $\phi^f$.

We make use of the following result of Sambarino {\cite[Lemma 2.4]{sambarino-quantitative}}.

\begin{lemma} 
\label{equib state and mme}
Suppose that  $f\in \Hol^\alpha(X,\Real)$ is positive. If $m_{-h_\phi(f)f}$ is the equilibrium state of $-h_\phi(f)f$, then
$${\rm d}m^\#=\frac{f {\rm d}m_{-h_\phi(f)f}}{\int f {\rm d}m_{-h_\phi(f)f}}$$ 
is the  measure of maximal entropy of $\phi^f.$ 
\end{lemma}

If $f,g\in \Hol^\alpha(X,\Real)$ are positive, we define their 
{\em pressure intersection}\footnote{We emphasize the terminology pressure intersection which is meant
to distinguish pressure intersection from the intersection defined by Bonahon \cite{bonahon}.}
by
\begin{equation}
\label{intersection and equib}
\II(f,g)=\lim_{T\to\infty} \frac{1}{\#R_\phi(f,T)}\sum_{a\in R_\phi(f,T)} \frac{\ell_{g}(a)}{\ell_{f}(a)} =\frac{\int g {\rm d}m_{-h(f)f}}{\int f {\rm d}m_{-h(f)f}}.
\end{equation}
The last equation follows from \cite[Sec. 3.4]{BCLS}.  We define the {\em renormalized pressure intersection} by
$$\JJ(f,g)=\frac{h_\phi(g)}{h_\phi(f)}\II(f,g).$$

In \cite[Cor. 2.5, Prop 3.11 and 3.12]{BCLS}, we used results of Parry-Pollicott \cite{parry-pollicott} and Ruelle \cite{ruelle-book} 
to prove the following.

\begin{proposition}
\label{hessianintersection}
If $\phi$ is a topologically transitive metric Anosov flow  on
a compact metric space $X$, then
\begin{enumerate}
\item 
If $f\in\Hol^\alpha(X,\Real)$ is positive, 
then the function $\JJ_{f}$ defined by $\JJ_f(g)=\JJ(f,g)$
has a global minimum at $f.$ 
Therefore, $\Hess\JJ_{f}$ is positive semi-definite. 
\item 
If $\{f_t\}_{t\in(-\epsilon,\epsilon)}\subset \Hol^\alpha(X,\Real)$ is a smooth one-parameter family of positive functions, 
then 
$$\frac{\partial^2}{\partial t^2}\Big|_{t=0}\JJ(f_0,f_t)=0$$ 
if and only if,  for every  $a\in O_\phi$, one has 
$$\frac{\partial}{\partial t}\Big|_{t=0} h_\phi(f_t) \ell_{f_t}(a)=0.$$
\item
If $\{f_u\}_{u\in M}$ and $\{g_v\}_{v\in M'}$ are analytic families of  positive $\alpha$-H\"older functions
parametrized by analytic manifolds $M$ and $M'$, then
$\JJ(f_u,g_v)$ is an analytic function on $M\times M'$.
\end{enumerate}
\end{proposition}

\subsection{Expansion on periodic orbits} 
\label{expansion}
Assume now that $X$ is a manifold and that $\phi$ is a $\clase^{1+\alpha}$ Anosov flow with unstable bundle $E^u.$ 
Denote by $\lambda^u_\phi:X\to(0,\infty)$ the \emph{infinitesimal expansion rate} on the unstable direction, defined by

$$\lambda^u_\phi(x)= \left.\frac\partial{\partial t}\right|_{t=0} \frac{1}{\kappa} \int_0^\kappa \log\det( d_x\phi_{t+s}|E^u)ds$$ for some $\kappa>0.$ 

We record the following observations (see \cite[Section 2.2]{potrie-sambarino} for further discussion):
\begin{enumerate}
\item 
If $a\in O_\phi$, then 
$$\ell_{\lambda_\phi^u}(a)=\int_a\lambda_\phi^u=\log\det (d_x\phi_{\ell(a)}|E^u)$$ 
is the total expansion of $\phi$ along $a.$
\item 
The Liv\v sic-cohomology class of $\lambda_\phi^u$ does not depend on $\kappa.$  
\item 
If $\phi^{-1}$ is the inverse flow $\phi^{-1}_t=\phi_{-t},$ it follows from Liv\v sic's Theorem (\cite{livsic}) that 
$\phi$ preserves a measure in the class of Lebesgue if and only if $\lambda^u_\phi$ is Liv\v sic 
cohomologuous to $\lambda^u_{\phi^{-1}}.$ 
\end{enumerate}

We make crucial use of the following classical result  of Sinai, Ruelle and Bowen.

\begin{theorem}[Sinai-Ruelle-Bowen \cite{bowen-ruelle}]\label{teo:srb} Let $\phi$ be a $\clase^{1+\alpha}$ Anosov flow 
on a compact manifold $X,$ then $\PP(-\lambda_\phi^u)=0.$ Moreover, if $\phi$ preserves a measure in the class 
of Lebesgue, then this measure is the equilibrium state of $-\lambda_\phi^u.$
\end{theorem}

Bowen and Ruelle state their result in the $\clase^2$ setting, but the proof may be extended to the
$\clase^{1+\alpha}$ setting by applying \cite[Prop. 19.16 and 20.4.2]{katok-hasselblatt}.

\subsection{Geodesic currents} 
\label{geodesic currents}

Let $\Gamma$ be a hyperbolic group which is not virtually cyclic. Let 
$\mathcal G(\Gamma)$ be the space of pair of distinct points, which we think of as the space of  {\em oriented geodesics},
on the Gromov boundary $\partial_\infty\Gamma$ of $\Gamma$:
$$
\mathcal G(\Gamma):=\{(x,y)\in\partial_\infty\Gamma\mid x\not=y\}\ .
$$

A {\em geodesic current} for $\Gamma$ is a $\Gamma$-invariant locally finite measure  on 
$\mathcal G(\Gamma)$. If $\gamma$ is a primitive infinite order element of $\Gamma$ with attracting fixed point 
$\gamma_+\in\partial_\infty\Gamma$ and repelling fixed point $\gamma_-$
and  $\delta_{(x,y)}$ is the Dirac measure supported at $(x,y)\in\mathcal G(\Gamma)$,  
we define the geodesic current
$$
 \delta_\gamma:=\sum_{\hat\gamma\in[\gamma]}\delta_{(\hat\gamma_-,\hat\gamma_+)}\ .
$$
where $[\gamma]$ is the conjugacy class of $\gamma$ in $\Gamma$.
If $\alpha=\gamma^n$ where $\alpha$
is primitive and $n>0$, we let $\delta_\alpha=n\delta_\gamma$. 

We let $\mathcal C(\Gamma)$ denote the space of  geodesic currents on $\Gamma$ and endow it with the 
\hbox{weak-*} topology. 
When $\Gamma=\pi_1(S)$, for a closed surface $S$, we write $\mathcal G(S)$ and $\mathcal C(S)$ 
for $\mathcal G(\pi_1(S))$ and $\mathcal C(\pi_1(S))$.

Following  Bonahon \cite[Section 4.2]{bonahon},  we define a
continuous, symmetric, bilinear pairing, called the {\em intersection}
$$i:\mathcal C(S)\times \mathcal C(S)\to \mathbb R$$
so that if $\alpha,\beta\in\Gamma$, then $i(\delta_\alpha,\delta_\beta)$ is the geometric intersection
number of the curves on $S$ representing $\alpha$ and $\beta$. 
Let $\mathcal{DG}(S)\subset \mathcal G(S)\times\mathcal G(S)$ denote the space of pairs $(x,y)$ and $(u,v)$ of
oriented geodesics which intersect, {\it {\it i.e. }} so that $x$ and $y$ lie in distinct components of $\partial_\infty\pi_1(S)-\{u,v\}$.
We then define
$$
i(\mu,\nu)=\int_{\mathcal{DG}(S)/\pi_1(S)}{\rm d\mu}\otimes{\rm d}\nu\ .
$$

A geodesic current is {\em symmetric} if it is invariant by the involution $\iota:(x,y)\mapsto (y,x)$.
Bonahon \cite{bonahon} works entirely in the setting of  symmetric geodesic currents. In fact, he defines a geodesic
current as a measure on the space $\widehat{\mathcal G}(\Gamma)={\mathcal G}(\Gamma)/\iota$ of unordered
pairs of distinct points in $\partial_\infty\Gamma$. A geodesic current $\mu$ in our sense naturally pushes forward to
a geodesic current $\hat \mu$ in the sense of Bonahon.  Moreover, if $\mu,\nu\in\mathcal C(S)$, then $i(\mu,\nu)$
agrees with the intersection, in the sense of Bonahon, of $\hat\mu$ and $\hat\nu$.

\subsection{Contracting line bundles and flows}
\label{contracting line bundles}

Gromov \cite{gromov} defined a geodesic flow ${\sf U}(\Gamma)$ for a hyperbolic group $\Gamma$, which is well-defined
up to H\"older orbit equivalence, see Champetier \cite{champetier} and Mineyev \cite{mineyev} for
detailed constructions.  The closed orbits of ${\sf U}(\Gamma)$ are in one-to-one correspondence with 
conjugacy classes of infinite order elements of $\Gamma$. There is a trivial H\"older 
$\Real$ principal bundle $\sf L_\Gamma=\tilde{\sf U}(\Gamma)$ over $\mathcal G(\Gamma)$ 
equipped with a properly discontinuous  action of $\Gamma$ by bundle automorphisms, 
so that $\sf L_\Gamma/\Gamma$ equipped with the flow coming from the action of 
$\Real$ is H\"older orbit equivalent to ${\sf U}(\Gamma)$.   Moreover, $\widetilde{\sf U}(\Gamma)$ may be 
parametrized as $\mathcal G(\Gamma)\times \mathbb R$ where the action of $\mathbb R$ is by translation in
the second factor.
We will mostly be interested in the situation where  $\sf{U}(\Gamma)$ is metric Anosov.

In \cite[Section 5]{BCLS}, we showed that,  whenever
a $\Gamma$ admits an Anosov representation, 
 $\sf{U}(\Gamma)$ is indeed  metric Anosov. In this paper, we will focus on the case where $\Gamma=\pi_1(S)$, 
in which case ${\sf U}(\Gamma)$ may be taken to be the geodesic flow on the unit tangent bundle of a hyperbolic
surface $Y$ homeomorphic to $S$, and will be denoted ${\sf U}(S)$, and $\sf L_\Gamma$ may
be identified with the geodesic flow on the unit tangent bundle  of the universal cover of $Y$, and will be denoted
${\sf U}(\tilde S)$.

A \emph{flow over $\mathcal G(\Gamma)$} is a H\"older $\Real$-principal line bundle ${\sf L}$ over $\mathcal G(\Gamma)$ 
equipped with a properly discontinuous action of $\Gamma$ by H\"older bundle automorphisms, so that the quotient flow on 
\hbox{$\sf U_{\sf L}:={\sf L}/\Gamma$} is H\"older orbit equivalent to the geodesic flow of $\Gamma$. 
In other words, one may think of a flow 
over $\mathcal G(\Gamma)$ as a parametrization of the geodesic flow of $\Gamma$.

Given a geodesic  current  $\omega$ and a flow ${\sf L}$ over $\mathcal G(\Gamma)$,  we define a pairing
$$ \braket{\omega\mid {\sf L}} :=\int_{{\sf U}_{\sf L}} \omega\otimes \d t\ $$ 
where $\d t$ is the element of arc length on ${\sf U}_{\sf L}$ given by the $\Real$-action.  
Given a flow $\sf L$ the function $\omega \mapsto \braket{\omega\mid {\sf L}}$ from 
$\mathcal C(\Gamma)$ to $\mathbb R$ is continuous.

We observe that, for every  non trivial  element $\gamma$ in $\Gamma$,
$\braket{\delta_\gamma\mid \sf L}$ is the length of the periodic orbit associated to $\gamma$ in ${\sf U}_{\sf L}$, 
or, equivalently,  the translation distance of the action of $\gamma$ on the fiber $\sf L_{(\gamma^-,\gamma^+)}$. 
The map $\gamma\mapsto \braket{\delta_\gamma\mid \sf L}$ is the {\em length spectrum} of ${\sf L}$.
If $\sf U(\Gamma)$ is metric Anosov,  then, by Liv\v sic's Theorem, the length spectrum determines the quotient flow ${\sf U}_{\sf L}$
up to H\"older conjugacy.

Let $\sf M$ be a H\"older line bundle over $\sf U(\Gamma)$, equipped with a lift of the geodesic flow 
$\{\psi_t\}_{t\in\mathbb R}$  on 
$\sf U(\Gamma)$ to a H\"older flow $\{\Psi_t\}_{t\in\mathbb R}$ on $\sf M$ by bundle automorphisms ({\it {\it i.e. }} the restriction of  $\Psi_t$
is a linear automorphism from ${\sf M}_z$ to ${\sf M}_{\psi_t(z)}$ for all $z\in {\sf U}(\Gamma)$ and all
$t\in\mathbb R$).
We  say that  $\sf M$ is {\em contracting} if there exist a metric $\| \cdot \|$ 
on $\sf M$  and $t_0>0$ so that
$$
\Vert \Psi_{t_0}(u)\Vert\le \frac{1}{2}\Vert u\Vert\ ,
$$
for all $u$ in $\sf M$. Every such line bundle has a {\em contraction spectrum} $\gamma\mapsto c(\gamma)$,
where  if the periodic orbit of $\sf U(\Gamma)$ associated to $\gamma\in\pi_1(S)$ has period $t_\gamma$, then
$$
\Vert\Psi_{t_\gamma}(v)\Vert=e^{-c(\gamma)}\Vert v\Vert\ 
$$
for any  vector $v$ in a fiber over the periodic orbit. Again Liv\v sic's Theorem guarantees that two line bundle with the same contracting spectrum are isomorphic.

The notions of contracting line bundles and flow are equivalent.

\begin{proposition}\label{line bundle to flow}
Let $\Gamma$ be a hyperbolic group whose geodesic flow is metric Anosov. Then	
\begin{enumerate}
	\item Given a contracting line bundle $\sf M$ over $\sf U(\Gamma)$, there exists a flow over $\mathcal G(\Gamma)$ whose length spectrum coincides with the contracting spectrum of $\sf M$.
	\item Conversely, given a flow $\sf L$ over $\mathcal G(\Gamma)$, there exists a contracting line bundle over $\sf U(\Gamma)$ whose contracting spectrum is the length spectrum of $\sf L$.
\end{enumerate}
\end{proposition}

\begin{proof} 
Given a contracting line bundle $\sf M$ over $\sf U(\Gamma)$, 
we construct a flow ${\sf L}_{\sf M}$ over $\mathcal G(\Gamma)$ by the following procedure
\begin{enumerate}
\item First,  lift  $\sf M$  to a line bundle $\widetilde{\sf M}$ over $\widetilde{\sf U}(\Gamma)$ 
and let $\{\widetilde\Psi_t\}_{t\in\mathbb R}$ be the lift of the flow $\{\Psi_t\}_{t\in\mathbb R}$ on $\sf M$.

\item We consider the corresponding 
$\Real$-principal line bundle $\widehat{\sf  L}_{\sf M}$ over $\widetilde{\sf U}(\Gamma)$ equipped with an 
action of $\Gamma$ by bundle automorphisms; concretely the fiber of  $\widehat{\sf  L}_{\sf M}$ over 
$(x,y,s)\in  \widetilde{\sf U}(\Gamma)$ is $(\widetilde{\sf M}_{(x,y,s)}-\{0\})/\pm 1$, 
{\it i.e. } non-zero vectors up
to sign, and the action of $t\in\mathbb R$ takes $[v]\in (\widehat{\sf  L}_{\sf M})_{(x,y,s)}$ to $[e^t v]$.
\item  Let $\pi:\tilde{\sf U}(\Gamma)\to\mathcal G(\Gamma)$. We  define ${\sf L}_{\sf M}:=\pi_*{\widehat{\sf L}_{\sf M}}$, that is the bundle whose sheaf of sections 
are the sections of $\widehat{\sf L}_{\sf M}$  invariant by the flow:
More explicitly, for all $t\in\mathbb R$, $(x,y,s)\in\widetilde{\sf U}(\Gamma)$
and $[v]\in(\widehat{\sf L}_{\sf M})_{(x,y,s)}$, we identify $[v]$ with $[\widetilde\Psi_t(v)]$ and notice that the quotient is a principal
$\mathbb R$-bundle over $\mathcal G(\Gamma)$.
\end{enumerate}

The proof of \cite[Prop. 4.2]{BCLS} generalizes immediately to yield the first part of our proposition.

We now establish our second claim. Let $\sf L$ be a flow over $\mathcal G(\Gamma)$. 
Consider the trivial bundle $\widetilde{\sf M}={\sf L}\times \Real$ 
over $\sf L$  equipped with the trivial lift of the action of $\Gamma$ given  by $\gamma(x,v)=(\gamma x, v)$.
Lift the  flow $\{\tilde\phi_t\}_{t\in\Real}$ on $\sf L$ to the flow
$$
\widetilde{\Psi_t}(x,v)=(\tilde\phi_t(x),e^{-t}v)
$$
on $\widetilde{\sf M}$.
These two  actions commute and we  obtain a contracting line bundle $\sf M :=\widetilde{\sf M}/\Gamma$ over $U_{\mathsf L}$ 
equipped with the quotient flow $\{\Psi_t\}_{t\in\Real}$  whose contracting spectrum agrees with the length spectrum of $\sf L$. 
\end{proof}

As an immediate consequence, the tensor product on principal $\mathbb R$ bundles  gives rise to an inner product,
also called the tensor product,
$$
(\sf L_0, \sf L_1)\mapsto \sf L_0\otimes \sf L_1\ ,
$$
on geodesic flows, which is equivalent to the tensor product of
the corresponding contracting line bundles. The length spectrum of the tensor product 
is then the sum of the two length spectra and thus for any geodesic current $\mu\in\mathcal C(\Gamma)$ 
$$
\braket{\mu\mid  \sf L_0\otimes \sf L_1}=\braket{\mu\mid  \sf L_0}+\braket{\mu\mid \sf L_1}\,
$$
since any current may be approximated by linear combinations of currents associated to group elements.
Given a positive number $t$, which we may view as an element of $\operatorname{Aut}(\mathbb R)$, 
we can renormalise the action of $\Real$ on the $\Real$-bundle $\sf L$ to obtain a new bundle $\sf L^t$ so that 
$$
\braket{\mu\mid  \sf L^t}=t\braket{\mu\mid  \sf L}\ .
$$
One can check then that for a positive integer $n$,
 $$
 \sf L^n={\overbrace{\sf L\otimes\ldots\otimes \sf L}^n}\ .
 $$

\section{Hitchin representations and their associated flows}

In Sections \ref{proj anosov} and \ref{subsection.Hitchinreps} we recall the definitions and basic properties of
projective Anosov and Hitchin representations. In Sections \ref{simple root flows} and \ref{spectral radius flows}
we use the techniques of Section \ref{contracting line bundles} to construct families of flows associated to
such representations.

\subsection{Projective Anosov representations}
\label{proj anosov}

It will occasionally be useful to work in the more general class of  {\em projective Anosov} representations.
A representation \hbox{$\rho:\Gamma\to \SL$} with domain a hyperbolic group $\Gamma$ has 
{\em transverse projective limit maps} 
if there exist continuous, \hbox{$\rho$-equivariant} functions
$$\xi_\rho:\partial_\infty\Gamma\to \mathbb P(\mathbb R^d)$$
and
$$\xi^*_\rho:\partial_\infty\Gamma\to \mathbb P((\mathbb R^d)^*)$$
so that if $x$ and $y$ are distinct points in $\partial_\infty\Gamma$, then 
$$\xi_\rho(x)\oplus\ker\xi^*_\rho(y)=\mathbb R^d.$$ 

Recall that a representation $\rho:\Gamma\to \sf{SL}(d,\mathbb R)$, with domain a hyperbolic group $\Gamma$,
gives rise to a flat $\mathbb R^d$-bundle $E_\rho$ over $\sf{U}(\Gamma)$ and
that the geodesic flow $\phi$ on $\sf{U}(\Gamma)$ lifts to a flow $\psi_\rho$  parallel to the flat connection on $E_\rho$.
Explicitly,  let  $\tilde E_\rho=\widetilde{\sf U}(\Gamma)\times \mathbb R^d$ and let $(\tilde\psi_\rho)_t(z,v)=(\tilde\phi_t(z),v)$
where $\tilde\phi_t$ is the lift of the geodesic flow $\phi_t$ on ${\sf U}(\Gamma)$
to $\widetilde{\sf U}(\Gamma)$. 
The group $\Gamma$ acts on $\tilde E_\rho$ by the action of $\Gamma$ on the first factor and $\rho(\Gamma)$
on the second factor, and the quotient is the flat bundle $E_\rho$ 
and the flow $\{(\tilde\psi_\rho)_t\}_{t\in\mathbb R}$ descends to a flow $\psi_\rho$ on $E_\rho$.

A representation $\rho$ with transverse projective limit maps determines a $\psi_\rho$-invariant splitting
$\Xi_\rho\oplus\Theta_\rho$ of the flat bundle $E_\rho$ over $\sf{U}(\Gamma)$. Concretely, 
the lift $\tilde \Xi_\rho$ of $\Xi_\rho$ has fiber
$\xi_\rho(x)$ and the  lift $\tilde \Theta_\rho$  of $\Theta_\rho$ has fiber $\ker\xi^*_\rho(y)$
over the point $(x,y,t)\in \widetilde{\sf U}(\Gamma)$.
One says that $\rho$ is {\em projective Anosov} if the resulting flow on the associated bundle
$${\rm Hom}(\Theta_\rho,\Xi_\rho)=\Xi_\rho\otimes\Theta_\rho^*$$ is contracting.

Projective Anosov representations are quasi-isometric embeddings with finite kernel, see \cite[Thm 5.3]{guichard-wienhard} and \cite[Thm 1.0.1]{labourie-energy} for Hitchin representations.
The following result is also a standard consequence of the definitions, see, for example, \cite[Prop. 2.6]{BCLS}).

\begin{lemma}\label{algo}
If $\rho:\pi_1(S)\rightarrow \SL$ is projective Anosov and $\gamma\in\pi_1(S)$ is non-trivial,
then $\rho(\g)$ is proximal with attracting line $\xi(\g_+)$ and repelling hyperplane $\theta(\gamma_{-})$. 
Moreover, there exist positive constants $B$ and $C$ such that 
$$\log\frac{\lambda_1(\rho(\gamma))}{\lambda_2(\rho(\gamma))}\ge B \ell(\gamma)-C$$
where $\ell(\gamma)$ is the reduced word length of $\gamma$.
\end{lemma}

\subsection{Hitchin  representations}
\label{subsection.Hitchinreps}

If $\rho:\pi_1(S)\to \psln$ is a Hitchin representation, it admits a lift $\tilde \rho:\pi_1(S)\to \SL$.
We will abuse notation and denote the flat bundle $E_{\tilde\rho}$ associated to this lift by $E_\rho$. (The flat
bundle depends on the choice of lift, but this choice will not matter for our purposes).

Labourie \cite{labourie-anosov} showed that every Hitchin representation $\rho$ admits a continuous $\rho$-equivariant 
limit map $\hat\xi_\rho:\partial_\infty\pi_1(S)\to \mathcal F_d$ where $\mathcal F_d$ is the space of complete flags in 
$\mathbb R^d$. We summarize its crucial properties below.

\begin{theorem}[{Labourie \cite{labourie-anosov}}]\label{hyperconvexity}
If $\rho \in \Hn$, there exists a unique $\rho$-equivariant H\"older continuous map
$\hat\xi_\rho: \partial_\infty\pi_1(S) \rightarrow {\mathcal F}_d$ such that
\begin{enumerate}
\item
If $d=n_1+\cdots+n_k$, where each $n_k\in\mathbb N$,  and $\{z_1,\ldots,z_k\}\subset\partial_\infty\pi_1(S)$ 
are pairwise distinct, then
$$\hat\xi_\rho^{(n_1)}(z_1)\oplus\cdots\oplus\hat\xi_\rho^{(n_k)}(z_k)=\mathbb R^d.$$
\item
The image  $\hat\xi_\rho^{(1)}( \partial_\infty\pi_1(S) )$ is a $\clase^{1+\alpha}$ manifold for some $\alpha>0$.
\item
The splitting $\bigoplus_{i=1}^d \tilde {\sf M}^i_\rho$ of $\widetilde E_\rho$ into line bundles so that 
$(\tilde{\sf  M}^i_\rho)_{(x,y,t)}=\hat\xi_\rho^{(i)}(x)\cap \hat\xi_\rho^{(n-i+1)}(y)$ descends to a splitting
$\bigoplus_{i=1}^d {\sf M}^i_\rho$ of $E_\rho$ into line bundles, so that ${\sf M}^i_\rho\otimes ({\sf M}^j_\rho)^*$ is contracting if $i<j$.
\end{enumerate}
\end{theorem}

It is well-known that  any exterior power of a (lift of a) Hitchin representation is projective Anosov 
(see for example Guichard-Wienhard \cite[Pop. 4.4]{guichard-wienhard}). Guichard has shown conversely in \cite{Guichard:2008} that the existence of such limit maps characterize Hitchin representations.

\begin{proposition}
\label{exterior is PA}
If $\rho\in\Hn$, $\tilde\rho:\pi_1(S)\to \sf{SL}(d,\mathbb R)$ is a lift of $\rho$, $k\in \{1,\ldots,d-1\}$, and ${\sf E}^k\tilde\rho:\pi_1(S)\to 
\sf{SL}(\Lambda^k\mathbb R^d)$
is the $k^{\rm th}$ exterior power of $\tilde\rho$, then ${\sf E}^k\tilde\rho$
is projective Anosov.
\end{proposition}

If $\rho\in\Hn$, then 
$$\xi_\rho=\xi_{\tilde\rho}=\hat\xi_\rho^{(1)}\qquad {\rm and}\qquad \ker\xi_\rho^*(x)=\ker\xi_{\tilde\rho}^*(x)=\hat\xi_\rho^{(d-1)}(x),$$
{\it i.e. }  $\xi_\rho^*(x)$ is the
projective class of  linear functionals with kernel $\hat\xi_\rho^{(d-1)}(x)$. More generally, if $\rho\in\Hn$, 
we may choose for each $x\in\partial_\infty\pi_1(S)$ a basis
$\{e_i(\rho,x)\}$ for $\mathbb R^d$ so that $\hat\xi_\rho^{(j)}(x)$ is spanned by $\{e_1(\rho,x),\ldots,e_j(\rho,x)\}$.
The limit maps for ${\sf E}^k\tilde\rho$ are given by
$$\xi_{{\sf E}^k\tilde\rho}(x)=\braket{e_1(\rho,x)\wedge\cdots\wedge e_k(\rho,x)}$$
and
$$\ker\xi^*_{{\sf E}^k\tilde\rho}(x)=\braket{e_{j_1}(\rho,x)\wedge\cdots\wedge e_{j_k}(\rho,x)\ | 
\ 1\le j_1<j_2<\cdots<j_k,\ j_k>k}.$$
One may check directly that
${\rm Hom}(\Theta_{{\sf E}^k\tilde\rho},\Xi_{{\sf E}^k\tilde\rho})$ is contracting and hence that ${\sf E}^k\tilde\rho$ is projective
Anosov, by applying part (3) of Theorem \ref{hyperconvexity}.

If we apply Lemma \ref{algo} to the exterior product ${\sf E}^i\tilde\rho$ of a (lift of a) Hitchin representation we obtain:

\begin{lemma}
\label{hitchin expansion}
If $\rho\in\Hn$ and $i\in\{1,\ldots,d-1\}$, then there
exist $B_i>0$ and $C_i$ so that
$$\log\frac{\lambda_i(\rho(\gamma))}{\lambda_{i+1}(\rho(\gamma))}\ge B_i\ell(\gamma)-C_i$$
where $\ell(\gamma)$ is the reduced word length of $\gamma$.
\end{lemma}

Lemma \ref{hitchin expansion} can also be derived directly from part (3) of Theorem \ref{hyperconvexity}.

\subsection{Flows for Hitchin representations}
\label{simple root flows}

Theorem \ref{hyperconvexity} provides several contracting line bundles  over ${\sf U}(S)$ associated to a Hitchin representation
$\rho\in\Hn$.
\begin{enumerate}
	\item The spectral radius line bundle ${\sf M}^1_\rho$.
	\item  The simple root line bundles ${\sf M}^{\alpha_i}_\rho:={\sf M}^i_\rho\otimes({\sf  M}^{i+1}_\rho)^*$.
	\item  The Hilbert line bundle ${\sf M}^{\sf H}_\rho:={\sf M}^1_\rho \otimes({\sf  M}_\rho^{d})^*$.
\end{enumerate}
Proposition \ref{line bundle to flow} shows that the associated flows 
\begin{enumerate}
\item
The spectral radius flow ${\sf U}_1(\rho)\ :=\ {\sf L}^1_\rho/\pi_1(S)$.
\item
The simple root flows ${\sf U}_{\alpha_i}(\rho)\ :=\ {\sf L}^{\alpha_i}_\rho/\pi_1(S)$.
\item 
The Hilbert flow ${\sf U}_{\sf H}(\rho)\ :=\  {\sf L}^{\sf H}_\rho/\pi_1(S)$.
\end{enumerate}
are all H\"older orbit equivalent to ${\sf U}(S)$. 
The corresponding length spectra are
\begin{enumerate}
\item
The spectral radius length $L_1(\rho(\gamma))\ :=\ \log(\lambda_1(\rho(\gamma)))$.
\item
The simple root length $L_{\alpha_i}(\rho(\gamma))\ :=\ \log(\frac{\lambda_i(\rho(\gamma))}{\lambda_{i+1}(\rho(\gamma))})$.
\item
The Hilbert length $L_{\sf H}(\rho(\gamma))\ :=\ \log(\frac{\lambda_1(\rho(\gamma))}{\lambda_{d}(\rho(\gamma))})$.
\end{enumerate}

More generally, given any positive linear combination $L_\phi=a_1L_{\alpha_1}+\ldots+a_{d-1}L_{\alpha_{d-1}}$
of the simple root length functions, we can find a flow ${\sf U}_\phi(\rho)$ so that the period of $\gamma\in\pi_1(S)$
is given by 
$$L_\phi(\rho(\gamma))=a_1L_{\alpha_1}(\rho(\gamma))+\ldots+a_{d-1}L_{\alpha_{d-1}}(\rho(\gamma)).$$
(See the discussion in Section \ref{contracting line bundles}.)

Finally, we observe that, by Theorem \ref{hyperconvexity} and  \cite[Prop. 6.2]{potrie-sambarino}, 
the flow  ${\sf M}^{1}_\rho$ is obtained as a pullback of a smooth line bundle over the $\clase^{1+\alpha}$-submanifold
$(\xi_\rho\times \xi^*_\rho)(\mathcal G(S))$ of 
$\mathbb P(\mathbb R^d)\times \mathbb P^*(\mathbb R^d)$, so ${\sf L}^1_\rho$ inherits  the structure of a $\clase^{1+\alpha}$-flow. 

Potrie and Sambarino \cite[Prop. 6.2]{potrie-sambarino} show that the unstable manifold $E_\rho^u$ for ${\sf L}^1_\rho$ at a point
above $(x,y)\in\mathcal G(S)$ may be identified with 
$\hom(\hat\xi_\rho^{(2)}(x)\cap\hat\xi_\rho^{(d-1)}(y),\hat\xi_\rho^{(1)}(x))$ and so
the {\em infinitesmal expansion rate} $\lambda_\rho^u$  of $\sf{U}_1(\rho)$ has the property that 
$$\int_\gamma \lambda_\rho^u {\rm d}s_\rho^1=L_{\alpha_1}(\rho(\gamma))$$
for all $\rho\in\Hn$,
where ${\rm d}s^1_\rho$ is the element of arc length of $\sf{U}_1(\rho)$. 

It follows that the reparametrization of $\sf{U}_1(\rho)$ by $\lambda_\rho^u$ is H\"older conjugate to 
$\sf{U}_{\alpha_1}(\rho)$. They then apply results of Sinai, Ruelle and Bowen \cite{bowen-ruelle},
to conclude that the entropy of $\sf{U}_{\alpha_1}(\rho)$ is 1. They further show, with a more sophisticated argument
in the general case, that all the simple root flows have entropy $1$.

\begin{theorem}{\rm (Potrie-Sambarino \cite[Thm. B]{potrie-sambarino})}
\label{entropy one}
If $\rho\in\Hn$ and $i\in\{1,\ldots,d-1\}$, then ${\sf U}_{\alpha_i}(\rho)$ has topological entropy 1.
\end{theorem}

\medskip\noindent
{\bf Remark:} One may also construct a flow H\"older conjugate to ${\sf U}_{\alpha_i}(\rho)$ by constructing,
as is done in Sambarino \cite{sambarino-quantitative}, a positive  H\"older function on ${\sf U}(S)$ whose periods are
given by $L_{\alpha_i}(\rho(\gamma))$, see also Potrie-Sambarino \cite{potrie-sambarino}.

\subsection{The spectral radius flow of a projective Anosov representation}
\label{spectral radius flows}

Proposition \ref{line bundle to flow} implies that if $\rho:\Gamma\to\SL$ is projective Anosov,
then the contracting line bundle $\Xi_\rho$ over ${\sf U}(\Gamma)$ gives rise to a spectral radius flow ${\sf L}_1^\rho$ 
over $\mathcal G(\Gamma)$ with
quotient ${\sf U}_1(\rho)$ so that the closed orbit associated to $\gamma\in\Gamma$ has period 
$L_1(\rho(\gamma))=\log(\lambda_1(\rho(\gamma))$.

The spectral radius flow ${\sf U}_1(\rho)$ is H\"older orbit equivalent to ${\sf U}(\Gamma)$. In \cite{BCLS}, we prove that, up to
H\"older conjugacy, the reparametrization function can be chosen to vary analytically in a neighborhood of $\rho$.

\begin{proposition}{\rm (\cite[Prop 6.2]{BCLS})}
\label{limit maps analytic}
Let $\{\rho_u:\pi_1(S)\to\sf{SL}(d,\mathbb R)\}_{u\in D}$ be a real analytic family of projective Anosov homomorphisms
parameterized by a disk $D$ about the origin 0. 
Then, there exists a sub-disk $D_0$ about 0, $\alpha>0$ and a real analytic family
$\{f_u : \sf{U}(\Gamma)\to \mathbb R\}_{u\in D_0}$ of positive $\alpha$-H\"older functions such that 
if $\gamma\in\Gamma$, then $\ell_{f_u}(\gamma)=\log\lambda_1(\rho_u(\gamma))$.
\end{proposition}

\section{Liouville currents for Hitchin representations}

In Sections \ref{sec:cr} and \ref{sec:defliouv} we recall Labourie's cross ratio, define our
Liouville current and prove that it determines the Hitchin representation.
In Section \ref{liouville basic prop} we establish relationships
between the Liouville current, Hilbert length $L_{\sf H}$, the Bowen-Margulis current $\mu_\rho$ for ${\sf U}_{\alpha_1}(\rho)$,
and the equilibrium state $m_{-\lambda^u_\rho}$ for the (negative of the) infinitesmal expansion rate on ${\sf U}_1(\rho)$.

\subsection{Labourie's cross ratio}\label{sec:cr}

If $V$ is a finite dimensional real vector space, let 
$$\P^{(2)}=\P(V)\times\P(V^*)-\{(L,\Phi):L\in\ker\Phi\}$$ 
and
$$\P^{(4)}=\{(L,\Phi,D,\Psi): L\notin\ker\Psi\textrm{ and }D\notin\ker\Phi\}\ .$$  
Consider the cross ratio on $\P^{(4)}$ defined by 
$$\B(L,\Phi,D,\Psi)=\frac{\varphi(u)}{\psi(u)}\frac{\psi(v)}{\varphi(v)},$$ 
where $\varphi\in\Phi,$ $\psi\in\Psi,$ $u\in L$ and $v\in D$ are all non-zero. 
Notice that the result does not depend on the choices of $\varphi,$ $\psi,$ $u$ and $v.$
Labourie observes that $\B$ is the polarized cross-ratio 
associated to a symplectic form on $\mathbb P^{(2)}$.

\begin{proposition}{\rm (Labourie \cite[Prop. 4.7, Prop. 5.4]{labourie-cross})}
\label{cross-symplectic} 
There exists a symplectic form $\Omega$ on $\mathbb P^{(2)}$ so that if
$(L,\Phi,D,\Psi)\in \P^{(4)}$, then 
$$\B(L,\Phi,D,\Psi)=e^{\int G^*\Omega}$$ 
where $G:[0,1]^2\to \P^{(2)}$ is a map such that the images of the vertices of $[0,1]^2$ are 
$(L,\Phi), (L,\Psi), (D,\Phi)$ and $(D,\Psi)$ and the image of every boundary segment is 
contained in either \hbox{$\P(V)\times\{\cdot\}$} or \hbox{$\{\cdot\}\times \P(V^*).$} 

Moreover, if $\rho$ is Hitchin, the restriction of  the symplectic form $\Omega$ to 
the $\clase^{1+\alpha}$-submanifold $(\xi_\rho\times\xi_\rho^*)(\mathcal G(S))$ is non-degenerate. 
\end{proposition}

Given $\rho\in\Hn$, Labourie defined a cross ratio $\b_\rho$ on 
$$\partial_\infty\pi_1(S)^{(4)}=\{(x,y,z,t)\in \partial_\infty\pi_1(S)^4\ | \ x\ne t,\ y\ne z\}$$
by setting 
$$\b_\rho(x,y,z,t)=\B\left(\xi_\rho(x),\xi_\rho^*(y),\xi_\rho(z),\xi_\rho^*(t)\right).$$
Labourie and McShane \cite[Thm. 9.0.3]{labourie-mcshane} show that 
$\b_\rho(x,z,t,y)>1$ if $(t,x,y,z)$ is cyclically ordered in $\partial_\infty\pi_1(S)$. 

Labourie \cite[Thm 1.1]{labourie-cross} proves that this cross ratio determines the representation and has rank $d$. 
For any pair  of $(p+1)$-tuples of pairwise distinct points $X= (x_0,\ldots,x_p)$ and $Y=(y_0,\ldots, y_p)$
in $\partial_\infty\pi_1(S)$, we define
$$
\chi_p(\b_\rho)(X,Y)=\det\Big(\b_\rho(x_i,y_j,x_0,y_0)\Big)_{i,j\in\{1,\ldots, p\}}\ .
$$

\begin{theorem}{\rm (Labourie \cite[Thm. 1.1]{labourie-cross})}
\label{cross ratio and rep}
If $\rho,\sigma\in\Hn$, then $\b_\rho=\b_\sigma$
if and only if $\rho=\sigma$. Moreover,  $\chi_d({\rm b}_\rho)\equiv 0$ and  $\chi_{d-1}({\rm b}_\rho)$ never vanishes. 
\end{theorem}

Labourie \cite[Thm. 1.1]{labourie-cross} also shows that the facts that $\chi_d({\rm b}_\rho)\equiv 0$ and  
$\chi_{d-1}({\rm b}_\rho)$ never vanishes
characterize cross ratios of Hitchin representations into $\psln$ among all $\pi_1(S)$-invariant functions  on $\partial_\infty\pi_1(S)^{(4)}$
satisfying the basic properties of a cross ratio.

\subsection{Liouville currents: basic definitions}\label{sec:defliouv}

Let $\omega_\rho$ be the geodesic current defined by 
$$\omega_\rho([t,x]\times [y,z])=\frac{1}{2}\log\b_\rho(x,z,t,y)>0$$
when $(x,y,z,t)$ is a cyclically ordered $4$-tuple in the circle $\partial_\infty\pi_1(S)$
and $[x,y]$ denotes the points between $x$ and $y$ in this cyclic ordering.

Proposition \ref{cross-symplectic} implies that
$$\omega_\rho([t,x]\times [y,z])=\frac{1}{2}\int_{\xi_\rho([t,x])\times\xi_\rho^*([y,z])}\Omega,$$
so $\omega_\rho$ is a measure on $\mathcal G(S)$ which is absolutely continuous with
respect to the Lebesgue measure obtained by identifying $\mathcal G(S)$ with the 
$\clase^{1+\alpha}$-manifold $(\xi_\rho\times\xi_\rho^*)(\mathcal G(S))$.
We call $\omega_\rho$ the {\em Liouville current}. 

We observe that the Liouville current also determines the Hitchin representation.

\medskip\noindent
{\bf Theorem \ref{Liouville current determines}.} {\em 
If $\rho\in\Hn$ and $\eta\in\mathcal H_m(S)$, then $\omega_\rho=\omega_\eta$ if and only if $\rho=\eta$.
}

\medskip\noindent
{\em Proof of Theorem \ref{Liouville current determines}.}
Suppose that $\omega_\rho=\omega_\eta$.
By definition, 
$$\b_\rho(x,y,z,t)=\omega_\rho([z,x]\times[t,y])=\omega_\eta([z,x]\times[t,y])=\b_\eta(x,y,z,t)$$
whenever $(z,x,t,y)$ is cyclically ordered.
Similarly, if $(z,x,y,t)$ is cyclically ordered, then
$$\b_\rho(x,y,z,t)=\frac{1}{\omega_\rho([z,x]\times[y,t])}=\frac{1}{\omega_\eta([z,x]\times[y,t])}=\b_\eta(x,y,z,t)$$
(One may summarize these two observations, by saying that $\b_\rho(x,y,z,t)=\b_\eta(x,y,z,t)$ whenever the
pairs $(x,z)$ and $(y,t)$ have non-intersecting axes, {\it i.e. } $y$ and $t$ lie in the same component of
$\partial_\infty\pi_1(S)-\{x,z\}$.)

Suppose that $m>d$. Let $X=(x_0,x_1,\ldots,x_{m-1})$ and $Y=(y_0,\ldots,y_{m-1})$  be two $m$-tuples in
$\partial_\infty\pi_1(S)$ so that $(x_0,x_1,\ldots,x_{m-1},y_0,y_1,\ldots,y_{m-1})$ is cyclically ordered.
It follows from the previous paragraph that $\b_\rho(x_i,y_j,x_0,y_0)=\b_\eta(x_i,y_j,x_0,y_0)$ for all $i,j>0$.
Theorem \ref{cross ratio and rep} then implies that every $(d+1)\times(d+1)$ minor of 
$\Big(\b_\rho(x_i,y_j,x_0,y_0)\Big)_{i,j\in\{1,\ldots, m-1\}}$ is zero, yet
$$\det\Big(\b_\rho(x_i,y_j,x_0,y_0)\Big)_{i,j\in\{1,\ldots, m-1\}}=\det\Big(\b_\eta(x_i,y_j,x_0,y_0)\Big)_{i,j\in\{1,\ldots, m-1\}}\ne 0$$
which is impossible. Therefore, we may assume that $m=d$.

By Theorem \ref{cross ratio and rep}, it suffices to prove that $\omega_\rho$ determines the cross-ratio 
$\b_\rho(x,y,z,t)$ of any 
$4$-tuple \hbox{$(x,y,z,t)\in\partial_\infty\pi_1(S)^{(4)}.$}
By the observations in the first paragraph, and symmetry,
it suffices to also consider the case where $(x,y,z,t)$ is cyclically ordered.

Fix a cyclically ordered configuration $(x_d,y_d,x_0,y_0)\in \partial_\infty\pi_1(S)^{(4)}.$
Choose pairwise distinct points $\{x_1,\ldots, x_{d-1}\}$ and $\{y_1,\ldots, y_{d-1}\}$ in $\partial_\infty\pi_1(S)$
so that \hbox{$(x_0,x_1,\ldots,x_{d-1}, y_{0},\ldots,y_{d-1},x_d,y_d)$} is cyclically ordered.
Let $X=(x_0,\ldots,x_d)$ and $Y=(y_0,\ldots,y_d)$. Theorem \ref{cross ratio and rep} implies
that 
$$\chi_d(\b_\rho)(X,Y))=\det\Big(\b_\rho(x_i,y_j,x_0,y_0)\Big)_{i,j\in\{1,\ldots, d\}}=
\det\Big(\b_\sigma(x_i,y_j,x_0,y_0)\Big)_{i,j\in\{1,\ldots, d\}}=\chi_d(\b_\eta)(X,Y)=0.$$
If $i$ and $j$ are not both $d$, then either $(x_0,x_i,y_0,y_j)$ or $(x_0,x_i,y_j,y_0)$ is cyclically ordered, so
$\b_\rho(x_i,y_j,x_0,y_0)=\b_\eta(x_i,y_j,x_0,y_0)$.
One sees that all the coefficients in the matrices above agree except for the term where $i=j=d$,
moreover, again applying Theorem \ref{cross ratio and rep}, we see that the minors
$$\det\Big(\b_\rho(x_i,y_j,x_0,y_0)\Big)_{i,j\in\{1,\ldots, d-1\}}=
\det\Big(\b_\sigma(x_i,y_j,x_0,y_0)\Big)_{i,j\in\{1,\ldots, d-1\}}\ne 0$$
agree and are non-zero. It follows that, $\b_\rho(x_d,y_d,x_0,y_0)=\b_\eta(x_d,y_d,x_0,y_0)$.
This completes the proof.
\eproof

\begin{corollary}
The Liouville current is  symmetric  if and only if $\rho=\rho^*$.
\end{corollary}

\begin{proof}
There is  a natural identification of $\P(\mathbb R^d)$ with $\P((\mathbb R^d)^*)$,
given by identifying $v \in \mathbb R^d$ to  the linear functional $w\to  v\cdot w$. So, given a representation $\rho\in\Hn$, 
$\xi_{\rho_*}=\xi_\rho^*$ and  $\xi^*_{\rho^*}=\xi_\rho$. 
Therefore, 
$$w_{\rho^*}([t,x]\times [y,z])=w_\rho([y,z]\times[t,x])=w_\rho \big(\iota ([t,x]\times [y,z])\big)$$
whenever $(x,y,z,t)$ is cyclically ordered. It follows that $\omega_\rho$ is symmetric if and only if
$\omega_\rho=\omega_{\rho_*}$. Theorem \ref{Liouville current determines} then completes the proof.
\end{proof}

\subsection{Liouville currents, equilibrium states and Bowen-Margulis measures}
\label{liouville basic prop}

We define the current
\begin{equation}
\label{def of mu}
\mu_\rho=\lim_{T\to\infty}\frac1{\#R_{\alpha_1}(\rho,T)} \sum_{[\g]\in R_{\alpha_1}(\rho,T)} \frac{1}{\braket{\g\mid {\sf L}^{\alpha_1}_\rho}}\delta_\g,
\end{equation}
where $R_{\alpha_1}(\rho,T)$ is the set of closed orbits of $\sf{U}_{\alpha_1}(\rho)$ of period at most $T$.
As was discussed in Section \ref{therm defs}, the measure of maximal entropy for $\sf{U}_{\alpha_1}(\rho)$ is 
the Bowen-Margulis measure for $\sf{U}_{\alpha_1}(\rho)$, given by
$$\mu_\rho\otimes {\rm d}s_\rho^{\alpha_1}=\lim_{T\to\infty}\frac1{\#R_{\alpha_1}(\rho,T)} \sum_{[\g]\in R_{\alpha_1}(\rho,T)} \frac{1}{\braket{\g\mid {\sf L}^{\alpha_1}_\rho}}\hat\delta_\g$$
where ${\rm d}s_\rho^{\alpha_1}$ is the element of arc length on $\sf{U}_{\alpha_1}(\rho)$. We
will refer to $\mu_\rho$ as the {\em Bowen-Margulis current} for $\sf{U}_{\alpha_1}(\rho)$.

The following result is an enlarged version of Theorem \ref{propertiesLiouville} from the introduction.

\begin{theorem}
\label{liouville and equib}
Suppose that $\rho\in\Hn$, $\omega_\rho$ is its Liouville current, $\lambda^u_\rho$ is the infinitesmal expansion rate of 
$\sf{U}_1(\rho)$ and $\mu_\rho$ is the Bowen-Margulis current for ${\sf U}_{\alpha_1}(\rho)$.

\begin{enumerate}
\item
If $\g\in\pi_1(S)$, then $i(\delta_\gamma,\omega_\rho)=L_{\sf{H}}(\rho(\gamma))=\braket{\delta_\gamma\mid {\sf L}^{\sf H}_\rho}$.
\item
If $\mu\in \mathcal C(S)$, then $i(\mu,\omega_\rho)=\braket{\mu\mid {\sf L}^{\sf H}_\rho}.$
\item
The equilibrium state $m_{-\lambda^u_\rho}$ for the H\"older potential $-\lambda^u_\rho$ on $\sf{U}_1(\rho)$
is a scalar multiple of $ \omega_\rho \otimes {\rm d}s_\rho^1$ where 
${\rm d}s_\rho^1$ is the element of arc length on $\sf{U}_1(\rho)$.
\item
The equilibrium state $m_{-\lambda^u_\rho}$ is a scalar multiple of $\mu_\rho\otimes {\rm d}s_\rho^1$.
\item 
The measure of maximal entropy for ${\sf U}_{\alpha_1}(\rho)$ is a scalar multiple of 
$\omega_\rho\otimes {\rm d}s_\rho^{\alpha_1}$.
\item
The Liouville current $\omega_\rho$ is a scalar multiple of the Bowen-Margulis current $\mu_\rho$. 
\end{enumerate}
\end{theorem}

\begin{proof}
A standard computation, see for example \cite[Prop 5.8]{labourie-cross}, shows that, for all \hbox{$\g\in\pi_1(S)$},
\begin{eqnarray*}
\label{intersection-curve}
i(\delta_\g,\omega_\rho)& = & \omega_\rho([\gamma_+,\gamma_-]\times[x,\gamma(x)]) +
\omega_\rho([\gamma_+,\gamma_-]\times[y,\gamma(y)])\\
&=  & \frac{1}{2}\big(\log\b_\rho(\gamma_-,\gamma(x),\gamma_+,x)+ 
\log\b_\rho(\gamma_-,\gamma(y),\gamma_+,y)\big)\\
& = & \log\frac{\lambda_1(\rho(\g))}{\lambda_d(\rho(\g))}\\
&=& L_H(\rho(\gamma))\\
&= & \braket{\delta_\gamma\mid {\sf L}^{\sf H}_\rho}
\end{eqnarray*} 
where $x$ and $y$ are in distinct components of  $\partial_\infty\pi_1(S)-\{\gamma_-,\gamma_+\}$.

Since every current is a limit of positive linear combinations of currents associated to elements of $\pi_1(S)$
and the intersection function is continuous in the  weak-* topology, we see that 
$$i(\mu,\omega_\rho)=\braket{\mu\mid \sf{L}^{\sf H}_\rho}$$
whenever $\mu\in\mathcal C(S)$.

Since $\omega_\rho$ is a measure on $\mathcal G(S)$ which is absolutely continuous with
respect to the pullback of the  Lebesgue measure on the 
$\clase^{1+\alpha}$-submanifold $(\xi_\rho\times\xi_\rho^*)(\mathcal G(S))$,
$\omega_\rho\otimes {\rm d}s_\rho^1$ is in the class of the Lebesgue measure
on the $\clase^{1+\alpha}$ manifold $\sf{U}_1(\rho)$. 
Theorem \ref{teo:srb} implies that  
$\omega_\rho\otimes {\rm d}s_\rho^1$ is a scalar multiple of the equilibrium state $m_{-\lambda^u_\rho}$ for $-\lambda_\rho^u$
on ${\sf U}_1(\rho),$ {\it i.e. }
\begin{equation}
\label{equib and omega}
m_{-\lambda^u_\rho}=\frac{ \omega_\rho\otimes {\rm d}s_\rho^1}{\braket{\omega_\rho\mid {\sf{L}}^1_\rho}}.
\end{equation}

Since $\sf{U}_{\alpha_1}(\rho)$ is H\"older conjugate to the reparametrization of $\sf{U}_1(\rho)$ by $\lambda_\rho^u$
and ${\sf U}_{\alpha_1}(\rho)$ has topological entropy 1,
the equilibrium measure $m_{-\lambda^u_\rho}$
is a scalar multiple of the pullback of  the measure of maximal entropy $\mu_\rho\otimes {\rm d}s_\rho^{\alpha_1}$ 
for ${\sf U}_{\alpha_1}(\rho)$ to $\sf{U}_1(\rho)$,  see Lemma \ref{equib state and mme}, {\it i.e. }
\begin{equation}
\label{equib and mu}
m_{-\lambda^u_\rho}=\frac{ \mu_\rho\otimes {\rm d}s_\rho^1}{\braket{{\mu}_\rho\mid {\sf L}^1_\rho}}.
\end{equation}

Since, by Equations (\ref{equib and omega}) and (\ref{equib and mu}),
$\mu_\rho\otimes {\rm d}s_\rho^1$ is a scalar multiple of $\omega_\rho\otimes {\rm d}s_\rho^1$, 
we see that $\mu_\rho$ is a scalar multiple of $\omega_\rho$. Therefore, the measure of
maximal entropy $\mu_\rho\otimes {\rm d}s_\rho^{\alpha_1}$ for ${\sf U}_{\alpha_1}(\rho)$
is a scalar multiple of $\omega_\rho\otimes {\rm d}s_\rho^{\alpha_1}$.

\end{proof}

As an immediate corollary, we obtain an expression for the intersection of two Liouville currents.

\begin{corollary} 
\label{intersection formula}
If $\rho\in{\mathcal H}_m(S)$ and $\eta\in\Hn$, then 
$$i(\omega_\rho,\omega_\eta)=\braket{\omega_\rho\mid {\sf L}^{\sf H}_\eta}\ .
$$
 \end{corollary}

\medskip\noindent
{\bf Remark:}   Since symmetric geodesic currents are determined by their periods \cite[Thm. 2]{otal},
our Liouville current $\omega_\rho$ pushes forward to Bonahon's
Liouville current on $\widehat{\mathcal G}(S)$, if $d=2$, and to the 
symmetric Liouville current defined by Martone and Zhang \cite{martone-zhang}, if $d>2$.

\section{Liouville volume rigidity}

Recall that we define the {\em Liouville volume} of $\rho\in\Hn$ by 
$$\vol_{\sf{L}}(\rho)=\braket{\omega_\rho\mid {\sf L}_\rho^{\sf H}}$$
so Corollary \ref{intersection formula} implies that
$$\vol_{\sf{L}}(\rho)=i(\omega_\rho,\omega_\rho).$$
In this section, we apply Corollary \ref{intersection formula},  an argument of Labourie \cite[Lemma 5.1]{labourie-fuchsian}
and a length spectrum rigidity result \cite[Theorem 11.2]{BCLS} to obtain a Liouville volume rigidity result.

\begin{theorem}
If $\rho,\eta\in\Hn$, then
$$
\frac{\vol_{\sf L}(\rho)}{\vol_{\sf L}(\eta)}\ge \left(\inf_{\gamma\in\pi_1(S)-\{1\}}\frac{L_{\sf{H}}(\rho(\gamma))}{L_{\sf{H}}(\eta(\gamma))}\right)^2\ .
$$
Moreover, equality holds if and only if either $\rho=\eta$ or $\rho=\eta^*$ where
$\eta^*$ is the contragredient of $\eta$.
\end{theorem}

Notice that, $\inf_{\gamma\in\pi_1(S)-\{1\}} \frac{L_{\sf{H}}(\rho(\gamma))}{L_{\sf{H}}(\eta(\gamma))}$ is finite and 
non-zero, since Hitchin representations are well-displacing (see  \cite[Thm. 6.1.3]{labourie-energy}). However,
if $d>2$, it can be arbitrarily close to 0 or $\infty$ (see Zhang \cite{zhang2}).

\begin{proof}
Let $K=\inf_{\gamma\in\pi_1(S)-\{1\}}\frac{L_{\sf{H}}(\rho(\gamma))}{L_{\sf{H}}(\eta(\gamma))}$ 
so that if $\gamma\in\pi_1(S)-\{1\}$, then
$$i(\delta_\gamma,\omega_\rho)=L_{\sf{H}}(\rho(\gamma))\ge K L_{\sf{H}}(\eta(\gamma))=K\ i(\delta_\gamma,\omega_\eta).$$
Since $\omega_\rho$ and $\omega_\eta$ are both limits of positive linear combinations of currents associated to 
elements of $\pi_1(S)$, this implies that
$$i(\omega_\rho,\omega_\rho)\ge K\ i(\omega_\rho,\omega_\eta)\quad{\rm and} \quad
i(\omega_\eta,\omega_\rho)\ge K\ i(\omega_\eta,\omega_\eta).$$ 
Therefore, using the fact that $i$ is symmetric,
$$\vol_{\sf{L}}(\rho)=i(\omega_\rho,\omega_\rho)\ge K\ i(\omega_\rho,\omega_\eta)=K\ i(\omega_\eta,\omega_\rho)\ge
K^2\ i(\omega_\eta,\omega_\eta)= K^2\vol_{\sf{L}}(\eta).$$

Now assume that, in addition, $\vol_{\sf{L}}(\rho)=K^2\vol_{\sf{L}}(\eta)$, so 
$$i(\omega_\rho,\omega_\rho)=K\ i(\omega_\rho,\omega_\eta)\quad {\rm and}
\quad i(\omega_\rho,\omega_\eta)=K\ i(\omega_\eta,\omega_\eta).$$
Since ${\sf U}_{\sf H}(\rho)$, ${\sf U}_{\sf H}(\eta)$ and ${\sf U}_{\alpha_1}(\eta)$ are all H\"older orbit
equivalent to ${\sf U}(S)$, we may assume that, up to H\"older conjugacy, there exist  positive H\"older functions
$g:{\sf U}(S)\to \mathbb R$ and $j:{\sf U}(S)\to \mathbb R$ so
that 
$$\d s_\rho^{\sf H}=g\d s_\eta^{\sf H}\qquad {\rm and}\qquad j {\rm d}s^{\alpha_1}_\eta={\rm d}s^{\sf H}_\eta.$$
So, applying Corollary \ref{intersection formula},
$$\int g {\rm d}\omega_\eta\otimes {\rm d}s^{\sf{H}}_\eta=\int {\rm d}\omega_\eta\otimes {\rm d}s^{\sf{H}}_\rho =
i(\omega_\eta,\omega_\rho)  =K \vol_{\sf{L}}(\eta)$$
and
\begin{equation}
\label{intg-kzero}
\int (g-K)\ {\rm d}\omega_\eta\otimes {\rm d}s_\eta^{\sf H}=
\int g\ {\rm d}\omega_\eta\otimes {\rm d}s_\eta^{\sf H}- K\int  {\rm d}\omega_\eta\otimes {\rm d}s_\eta^{\sf H} =
K\vol_{\sf{L}}(\eta)- K\vol_{\sf{L}}(\eta)=0.
\end{equation}
On the other hand, since $L_{\sf{H}}(\rho(\gamma))\ge K L_{\sf{H}}(\eta(\gamma))$,
$$\int_\gamma (g-K)\ {\rm d}s_\eta^{\sf H}=L_{\sf{H}}(\rho(\gamma))-K L_{\sf{H}}(\eta(\gamma))\ge 0$$
for all $\gamma\in\pi_1(S)-\{1\}$.

Let $f=(g-K)j$.
We will apply the argument of \cite[Lemma 5.1]{labourie-fuchsian} to establish our rigidity claim.
If $\gamma\in\pi_1(S)$, then
$$\int_\gamma f {\rm d}s_\eta^{\alpha_1}=\int_\gamma (g-K)\ {\rm d}s_\eta^{\sf{H}}\ge 0.$$
Since measures supported on periodic orbits are dense in the space $\mathcal M_{\sf{U}_{\alpha_1}(\eta)}$ 
of all flow invariant probability measures on $\sf{U}_{\alpha_1}(\eta)$ (see Sigmund \cite{sigmund}), 
we see that 
\begin{eqnarray}
\int f {\rm d}\mu\ge 0\ ,	\label{ineq:pres1}
\end{eqnarray}
for all $\mu\in\mathcal M_{\sf{U}_{\alpha_1}(\eta)}$.   

Since $\mu_\eta$ is a multiple of $\omega_\eta$, equation (\ref{intg-kzero}) implies that
\begin{equation}
\label{eq:pres2}\int f\  {\rm d}\mu_\eta\otimes {\rm d}s_\eta^{\alpha_1}  = \int (g-K)j\ {\rm d}\mu_\eta\otimes {\rm d}s_\eta^{\alpha_1} = \int (g-K) \ {\rm d}\mu_\eta\otimes {\rm d}s_\eta^{\sf H} =  0,
\end{equation}
so
\begin{eqnarray*}
	\sup_{\mu\in\mathcal M_{\sf{U}_{\alpha_1}(\eta)}}\left(h(\mu)-\int f{\rm d}\mu\right) & 
	\leq&\sup_{\mu\in\mathcal M_{\sf{U}_{\alpha_1}(\eta)}}  h(\mu)\cr
&=&h(\mu_\eta\otimes {\rm d}s_\eta^{\alpha_1})\cr
	&=&h(\mu_\eta\otimes {\rm d}s_\eta^{\alpha_1})-\int f\  {\rm d}\mu_\eta\otimes {\rm d}s_\eta^{\alpha_1}\cr
	&\leq&\sup_{\mu\in\mathcal M_{\sf{U}_{\alpha_1}(\eta)}}\left(h(\mu)-\int f{\rm d}\mu\right)\cr
\end{eqnarray*}
where the first inequality follows  from inequality \eqref{ineq:pres1}, the equality in the second line holds because 
$\mu_\eta\otimes {\rm d}s_\eta^{\alpha_1}$ is the measure of maximal entropy
for $\sf{U}_{\alpha_1}(\eta)$, the equality in the third line follows from equation (\ref{eq:pres2})
and the final inequality holds by definition.
Therefore,
$${\PP}(-f)=\sup_{\mu\in\mathcal M_{\sf{U}_{\alpha_1}(\eta)}}\left(h(\mu)-\int f{\rm d}\mu\right) =h(\mu_\eta\otimes {\rm d}s_\eta^{\alpha_1})-\int f\  {\rm d}\mu_\eta\otimes {\rm d}s_\eta^{\alpha_1}\ ,$$
so $\mu_\eta\otimes {\rm d}s_\eta^{\alpha_1}$ is the equilibrium state for $-f$. Since 
$\omega_\eta\otimes {\rm d}s_\eta^{\alpha_1}$ is also the equilibrium state for the zero
function, \cite[Prop 20.3.10]{katok-hasselblatt} implies that
$-f$ is Liv\v sic cohomologuous to a constant function $A$. However, $A=0$ since 
$$\int f {\rm d}\omega_\eta\otimes {\rm d}s_\eta^{\alpha_1}= 0.$$

It follows that for all $\gamma\in\pi_1(S)$, 
$$L_{\sf H}(\rho(\gamma))-KL_{\sf H}(\eta(\gamma))=\int_\gamma f {\rm d}s^{\alpha_1}_\eta =0\ .$$
Therefore, $L_{\sf{H}}(\rho(\gamma))=K L_{\sf{H}}(\eta(\gamma))$ for all $\gamma\in\pi_1(S)$.

We recall that since $\rho$ and $\sigma$ are projective Anosov, ${\rm Ad}\rho$ and ${\rm Ad}(\sigma)$ are
also projective Anosov (see \cite[Section 10.2]{guichard-wienhard}). 
Since $\lambda_1({\rm Ad}\rho(\gamma))=L_{\sf H}(\rho(\gamma))$ for all $\gamma\in\pi_1(S)$, 
\cite[Theorem 11.2]{BCLS} implies that $K=1$ and either ${\rm Ad}\rho={\rm Ad}\eta$ or ${\rm Ad}\rho={\rm Ad}\eta^*$.
Therefore, either $\rho=\eta$ or $\rho=\eta^*$. (When $d=3$, we could  apply earlier results of Cooper-Delp
\cite{cooper-delp} or Kim \cite{kim-projective}.)
\end{proof}

We obtain the following corollary, stated in the introduction as Theorem \ref{thurston metric version}, by symmetry.

\begin{corollary}
\label{thurston metric version2}
If $\rho,\eta\in\Hn,$  then
$$\left(\inf_{\gamma\in\pi_1(S)\setminus\{1\}}\frac{L_{\sf{H}}(\rho(\gamma))}{L_{\sf{H}}(\eta(\gamma))}\right)^2\le \frac{\vol_{\sf{L}}(\rho)}{\vol_{\sf{L}}(\eta)} \le \left(\sup_{\gamma\in\pi_1(S)\setminus\{1\}} \frac{L_{\sf{H}}(\rho(\gamma))}{L_{\sf{H}}(\eta(\gamma))}\right)^2$$
and equality holds in either inequality if and only if  either $\rho=\eta$ or $\rho=\eta^*$.
\end{corollary}

If $\rho\in \mathcal H_3(S)$, Tholozan \cite[Thm. 3]{tholozan} showed that there exists a $3$-Fuchsian representation
$\sigma=\tau_3\circ\sigma_0$, where $\sigma_0:\pi_1(S)\to \sf{PSL}(2,\mathbb R)$ is Fuchsian and 
$\tau_3:\sf{PSL}(2,\mathbb R)\to \psln$ is the irreducible representation, 
so that $\rho$ dominates $\sigma$, {\it i.e. }
$L_\sf{H}(\rho(\gamma))\ge L_\sf{H}(\sigma(\gamma))$ for all $\gamma\in\pi_1(S)$.
Since $\omega_\sigma=2\omega_{\sigma_0}$ and $i(\omega_{\sigma_0},\omega_{\sigma_0})=\pi^2|\chi(S)|$
(see Bonahon \cite[Prop. 15]{bonahon}),
Corollary \ref{thurston metric version2} implies that \hbox{$\vol_{\sf{L}}(\rho)\ge \vol_{\sf{L}}(\sigma)=4\pi^2|\chi(S)|$}.

\begin{corollary}
If $\rho\in\mathcal H_3(S)$, then
$$\vol_{\sf{L}}(\rho)\ge 4\pi^2|\chi(S)|.$$
Moreover, equality holds if and only if $\rho$ is $3$-Fuchsian.
\end{corollary}

If $\rho\in\mathcal H_3(S)$, then, see Choi-Goldman \cite{choi-goldman}, there exists a strictly convex open domain
$\Omega_\rho$ in $\mathbb{RP}^2$ so that $\rho(\pi_1(S))$ acts properly discontinuously and cocompactly
on $\Omega_\rho$. It would be interesting to explore the relationship between $\vol_{\sf{L}}(\rho)$ and
other notions of volume for $\Omega_\rho/\rho(\pi_1(S))$.

\medskip

If $\sigma=\tau_d\circ\sigma_0\in \Hn$ is $d$-Fuchsian, 
then $\omega_\sigma=(d-1)\omega_{\sigma_0}$, so $\vol_{\sf{L}}(\sigma)=(d-1)^2\pi^2|\chi(S)|$.
It is known that not every $\rho\in\Hn$ dominates a Fuchsian representation, but one might still 
ask the following question.

\medskip\noindent
{\bf Question:} {\em Is it true that, for all $d>3$, 
$$\vol_{\sf{L}}(\rho)\ge (d-1)^2\pi^2|\chi(S)|$$
for all $\rho\in \Hn$? If so, does equality hold if and only if $\rho$ is $d$-Fuchsian?
}

\section{Pressure quadratic forms associated to simple roots}

In \cite[Section 3]{BCS-survey}, we describe a general procedure for producing pressure metrics on
deformation spaces of representations based on the constructions in McMullen \cite{mcmullen-pressure}, 
Bridgeman \cite{bridgeman-wp}
and \cite{BCLS}. The first step in the process is to associate a flow to each representation. One then defines
an associated pressure intersection and renormalized pressure intersection. Fundamental properties
from the thermodynamic formalism, as summarized in Proposition \ref{hessianintersection}, then guarantee
that the Hessian of the renormalized intersection gives rise to a non-negative quadratic form on the tangent
space to the deformation space. The resulting quadratic form may or may not be positive definite and
the analysis of its degeneracy is typically the most difficult step in this procedure.

Recall that,  in Section \ref{simple root flows},
we associated a family ${\sf U}_{\alpha_i}(\rho)$ of simple root flows to a Hitchin representation.
We interpret the next result to say that this family of flows varies analytically over the Hitchin component.

\begin{proposition}
\label{analytic lifts}
For all $i\in\{1,\ldots, d-1\}$ and  $\rho\in\Hn$, 
there exists a neighborhood $V_i$ of $\rho$ in $\Hn$, $\nu_i>0$
and an analytic map $T_i:V_i\to \mathcal \Hol^{\nu_i}(\sf{U}(S))$ such that if $\sigma\in V_i$,
then $T_i(\sigma)$ is positive and  $\ell_{T_i(\sigma)}(\gamma)=L_{\alpha_i}(\sigma(\gamma))$ for all $\gamma \in \pi_1(S)$.
\end{proposition}

Notice that the conclusion  of Proposition \ref{analytic lifts} implies that the reparametrization of ${\sf U}(S)$
by $T_i(\sigma)$ is H\"older conjugate to ${\sf U}_{\alpha_i}(\sigma)$.

\begin{proof}
Let $\rho\in\Hn$.
Proposition \ref{limit maps analytic} implies that there exists a neighborhood $W_1$ of $\rho$, $\beta_1>0$ 
and an analytic map 
$S_1:W_1\to \mathcal \Hol^{\beta_1}(\sf{U}(S))$, so that
$$\ell_{S_1(\sigma)}(\gamma)=\log\lambda_1(\sigma(\gamma))$$
for all $\gamma\in \pi_1(S)$ and $\sigma\in W_1$.
Similarly, since for all $i\in\{2,\ldots,d-1\}$, the exterior power ${\sf E}^i\tilde\rho$  of a lift of $\rho$ is projective Anosov, 
by Proposition \ref{exterior is PA},
Proposition \ref{limit maps analytic} implies that
there exists a neighborhood $W_i$ of $\rho$ in $\Hn$, $\beta_i>0$
and an analytic map $S_i:W_i\to \Hol^{\beta_i}(\sf{U}(S),\Real)$
so that if $\sigma\in W_i$, then 
$$\ell_{S_i(\sigma)}(\gamma)=\log\lambda_1(E^i\tilde\sigma(\gamma))=
\log\big(\lambda_1(\sigma(\gamma))\lambda_2(\sigma(\gamma))\cdots
\lambda_i(\sigma(\gamma))\big)$$
for all $\gamma\in\pi_1(S)$.

Let $\widehat V_1=W_1\cap W_2$ and $\hat\nu_1=\min\{\beta_1,\beta_2\}$ and
define  an analytic map \hbox{$\widehat T_1:V_1\to \Hol^{\hat\nu_1}(\sf{U}(S))$} 
by setting $\widehat T_1(\sigma)=2S_1(\sigma)-S_2(\sigma)$.
Then 
$$\ell_{\widehat T_1(\sigma)}(\gamma)=2\log\lambda_1(\sigma(\gamma))-\log\lambda_1(E^2\sigma(\gamma))=
\log\left(\frac{\lambda_1(\sigma(\gamma))}{\lambda_2(\sigma(\gamma))}\right)=L_{\alpha_1}(\sigma(\gamma))$$
for all $\gamma\in\pi_1(S)$ and $\sigma\in \widehat V_1$. 

More generally, if $i\in\{2,\ldots,d-2\}$, let
$\widehat V_i=W_1\cap W_2\cap\cdots\cap W_{i+1}$ and
$\hat\nu_i=\min\{\beta_1,\ldots,\beta_{i+1}\}$, and
define $\widehat T_i:\widehat V_i\to \Hol^{\hat\nu_i}(\sf{U}(S))$ by setting 
$$\widehat T_i(\sigma)=2S_i(\sigma)-S_{i+1}(\sigma)-S_{i-1}(\sigma).$$
One easily checks that  $\ell_{\widehat T_i(\sigma)}(\gamma)=L_{\alpha_i}(\sigma(\gamma))$
for all $\gamma \in \pi_1(S)$ and $\sigma\in\widehat V_i$.
Finally, we define \hbox{$\widehat T_{d-1}:\widehat V_1\to \Hol^{\hat\nu_{d-1}}(\sf{U}(S))$}, where $\hat\nu_{d-1}=\hat\nu_1$,
by \hbox{$\widehat T_{d-1}(\sigma)=\widehat T_1(\sigma)\circ F$}
where \hbox{$F:\sf{U}(S)\to \sf{U}(S)$} is given by $F( v)=- v$, and check that 
\hbox{$\ell_{\widehat T_{d-1}(\sigma)}(\gamma)=L_{\alpha_1}(\sigma(\gamma^{-1}))=L_{\alpha_{d-1}}(\sigma(\gamma))$}
for all $\gamma \in \pi_1(S)$ and $\sigma\in\widehat V_{d-1}=\widehat V_1$

It remains to alter each $\widehat T_i$ so that, after restricting to a sub-neighborhood of $\widehat V_i$, 
the image consists of positive functions.
Since $\widehat T_i(\rho)$ has positive periods, it is Liv\v sic 
cohomologous to a positive \hbox{$\tau_i$-H\"older} function $f_i$, 
for some $\tau_i>0$ (see \cite[Lemma 3.8]{sambarino-quantitative}).
Define \hbox{$T_i:\widehat V_i\to  \Hol^{\nu_i}(\sf{U}(S))$}, where $\nu_i=\min\{\hat\nu_i,\tau_i\}$, by setting
$$T_i(\sigma)=\widehat T_i(\sigma)+(f_i-\widehat T_i(\rho)).$$
We now check that $T_i$ has the properties we claimed.
\begin{enumerate}
	\item Since $\widehat T_i$ is analytic, and 
	$T_i$ is a translate of $\widehat T_i$, $T_i$ is also analytic.
	\item Since $f_i-\widehat T_i(\rho)$ is Liv\v sic
cohomologous to 0, $T_i(\sigma)$ is Liv\v sic cohomologous to $\widehat T_i(\sigma)$. In particular, they have the same periods,
so 
\hbox{$\ell_{T_i(\sigma)}=\ell_{\widehat T_i(\sigma)}(\gamma)=L_{\alpha_i}(\sigma(\gamma))$}
for all $\gamma\in\pi_1(S)$ and $\sigma\in \widehat V_1$. 
\item Since ${\sf U}(S)$ is compact, the set of positive functions is an open subset of $\Hol^{\nu_i}(\sf{U}(S))$.
Since $T_i(\rho)$ is a positive function and 
$T_i$ is analytic, hence continuous,  there is a neighbourhood $V_i\subset\widehat V_i$ of $\rho$
so that $T_i(\sigma)$ is  a positive function for all $\sigma\in V_i$.
\end{enumerate}

\end{proof}

We then define the pressure intersection
$$\II_{\alpha_i}(\rho,\eta)=\lim_{T\to\infty}\frac1{\# R_{\alpha_i}(\rho,T)}\sum_{\g\in R_{\alpha_i}(\rho,T)}\frac{ L_{\alpha_i}(\eta(\g))}{ L_{\alpha_i}(\rho(\g))}=\II(f_\rho^i,f_\eta^i)
$$
for all $\rho,\eta\in\Hn$, where
$$R_{\alpha_i}(\rho,T)=\{[\gamma]\in [\pi_1(S)]\setminus\{[1]\}\ |\ L_{\alpha_i}(\rho(\gamma))\le T\}$$
and the reparametrizations of ${\sf U}(S)$ by $f_\rho^i$ and $f_\eta^i$ are H\"older conjugate to ${\sf U}_{\alpha_i}(\rho)$
and ${\sf U}_{\alpha_i}(\eta)$.
For fixed $\rho\in \Hn$,
we further define $(\II_{\alpha_i})_\rho:\Hn\to\mathbb R$ by
$$(\II_{\alpha_i})_\rho(\sigma)=\II_{\alpha_i}(\rho,\sigma)$$ 
for all $\sigma\in \Hn$. 
If $V_i$ is the neighorhood of $\rho$ and $T_i$ is the map provided by Proposition \ref{analytic lifts}, then 
$$\II_{\alpha_i}(\sigma,\eta)=\II(T_i(\sigma),T_i(\eta))$$
for all $\sigma,\eta\in V_i$.
By Theorem \ref{entropy one}, ${\sf U}_{\alpha_i}(\sigma)$ has entropy 1, for all $\sigma \in \Hn$ and all $i$,
so 
$$\II_{\alpha_i}(\sigma,\eta)=\II(T_i(\sigma),T_i(\eta))=\JJ(T_i(\sigma),T_i(\eta))$$
for all $\sigma,\eta\in V_i$.
Proposition \ref{hessianintersection} then implies that
$$\PP_{\alpha_i}|_{\sf{T}_\rho\Hn}=\Hess_\rho(\II_{\alpha_i})_\rho$$
is positive semi-definite and varies analytically over $\Hn$.

By construction, the extended mapping class group and the contragredient 
preserve each  $\PP_{\alpha_i}$. It follows immediately from work of Wolpert \cite{wolpert},
that the restriction of each $\PP_{\alpha_i}$ to the Fuchsian locus is a positive multiple of the Weil-Petersson metric.
Since $L_{\alpha_i}(\rho(\gamma))=L_{\alpha_{d-i}}(\rho(\gamma^{-1}))$ for all $\rho\in\Hn$ and $\gamma\in\pi_1(S)$,
we see that \hbox{$\II_{\alpha_i}(\rho,\sigma)=\II_{\alpha_{d-i}}(\rho,\sigma)$} for all $\rho,\sigma\in\Hn$, so
\hbox{$\PP_{\alpha_i}=\PP_{\alpha_{d-i}}$} for all $i.$ 

We combine these observations with the non-degeneracy criterion provided by
Proposition \ref{hessianintersection} to obtain:

\begin{proposition}
\label{pressure forms}
For each $i\in\{1,\ldots,d-1\}$, there exists a positive semi-definite, analytic, quadratic form $\PP_{\alpha_i}$ on
${\sf T}\Hn$, which is invariant under the action of the mapping class group and restricts to a multiple
of the Weil-Petersson metric on the Fuchsian locus. Moreover, if $\{\rho_t\}_t\in(-\epsilon,\epsilon)$ is a smooth 
one-parameter family in $\Hn$,
then $\|\dot \rho_0\|^2_{\PP_{\alpha_i}}=0$ if and only if 
$$\frac{\partial}{\partial t}\Big|_{t=0} \braket{\gamma\mid {\sf L}^{\alpha_i}_{\rho_t}}= 
\frac{\partial}{\partial t}\Big|_{t=0}  L_{\alpha_i}({\rho_t}(\gamma))= 0$$
for all $\gamma\in \pi_1(S)$.
\end{proposition}

\noindent
{\bf Remark:} Labourie and Wentworth \cite{labourie-wentworth} evaluate the original pressure metric at the Fuchsian locus.
They remark \cite[Sec. 6.6]{labourie-wentworth} that their analysis should extend to the pressure 
quadratic forms $\PP_{\alpha_i}$.

\medskip

Finally, we observe that, as was claimed in the introduction, we may rewrite the Liouville pressure intersection $\II_{\alpha_1}$ as 
$$
\II_{\alpha_1}(\rho,\eta)=\frac{1}{\braket{\omega_\rho\mid {\sf  L}^{\alpha_1}_\rho}}\braket{\omega_\rho\mid {\sf L}^{\alpha_1}_\eta}.
$$ 
Notice that, by Theorem \ref{liouville and equib}, $\omega_\rho$ is a scalar  multiple of $\mu_\rho$ so
$\omega_\rho = c_\rho\mu_\rho$ for some $c_\rho\in\mathbb R$.
Since $\mu_\rho=\frac{1}{\#R_{\alpha_1}(\rho,T)} \sum_{R_{\alpha_1}(\rho,T)} \frac{\delta_\gamma}{L_{\alpha_1}(\rho(\gamma))}$,
we see that $\braket{\omega_\rho\mid {\sf L}^{\alpha_1}_\rho} = c_\rho$ and
$$\braket{\omega_\rho\mid {\sf  L}^{\alpha_1}_\eta} = c_\rho \lim_{T \rightarrow \infty}  \frac{1}{\#R_{\alpha_1}(\rho,T)} \sum_{R_{\alpha_1}(\rho,T)} 
\frac{\braket{\delta_\gamma\mid {\sf L}^{\alpha_1}_\eta}}{L_{\alpha_1}(\rho(\gamma))} =
 c_\rho \lim_{T \rightarrow \infty}  \frac{1}{\#R_{\alpha_1}(\rho,T)} \sum_{R_{\alpha_1}(\rho,T)} 
\frac{L_{\alpha_1}(\eta(\gamma))}{L_{\alpha_1}(\rho(\gamma))}$$
Therefore,
$$\frac{1}{\braket{\omega_\rho\mid {\sf L}^{\alpha_1}_\rho}}\braket{\omega_\rho\mid {\sf L}^{\alpha_1}_\eta} =
  \lim_{T \rightarrow \infty}  \frac{1}{\#R_{\alpha_1}(\rho,T)} \sum_{R_{\alpha_1}(\rho,T)} 
\frac{L_{\alpha_1}(\eta(\gamma))}{L_{\alpha_1}(\rho(\gamma))}=
\II_{\alpha_1}(\rho,\eta).$$

\section{The Liouville pressure quadratic form is a Riemannian metric}\label{non-degenerate}

The main work of this section is to show that the derivatives of the $L_{\alpha_1}$-length functions generate the
cotangent space of the Hitchin component.

\medskip\noindent
{\bf Theorem \ref{cotangent intro}.} {\em
If $\rho\in\Hn$, then the set 
$$\{{\rm D}_\rho L^\gamma_{\alpha_1}\}_{\gamma\in\pi_1(S)}$$ 
generates the cotangent space $\TT^*_\rho\Hn.$
}

\medskip

Theorem \ref{cotangent intro} and Proposition \ref{pressure forms} together imply that the Liouville
pressure quadratic form is a Riemannian metric.

\medskip\noindent
{\bf Theorem \ref{main}.} {\em
The Liouville pressure quadratic form $\PP_{\alpha_1}$ 
is a mapping class group invariant, analytic Riemannian metric on $\Hn$,
that restricts to a scalar multiple of the the Weil-Petersson metric on the Fuchsian locus.
}

\medskip\noindent
{\em Proof of Theorem \ref{main}:}
Suppose that $v\in {\sf T}_\rho\Hn$ and $\|v\|_{\PP_{\alpha_1}}=0$. Proposition \ref{pressure forms} implies
that ${\rm D}_\rho L_{\alpha_1}^\gamma(v)=0$ for all $\gamma \in\pi_1(S)$. Theorem \ref{cotangent intro} then implies
that $v=0$. Therefore,  since we already know it is positive semi-definite, $\PP_{\alpha_1}$ is positive definite. 
The remainder of the theorem follows from
Proposition \ref{pressure forms}.
\eproof

The remainder of the section will be taken up with the proof of Theorem \ref{cotangent intro}.
Theorem \ref{cotangent intro} generalizes \cite[Prop. 10.1]{BCLS}, which asserts that derivatives of the
spectral radius functions generate
the cotangent space, and its proof follows a similar outline.
We use an analysis of the asymptotic behavior of the $L_{\alpha_1}$-length functions to show that if
${\rm D}_\rho L^\gamma_{\alpha_1}(v) = 0$  for all $\gamma$, then the derivatives of functions which record the eigenvalues are
also trivial in the direction $v$.  We then apply \cite[Prop. 10.1]{BCLS} itself to finish the proof, but we could also have observed
that the derivatives of all trace functions are trivial in the direction $v$ and applied standard facts about character varieties.

\subsection{Transversality results}

Let $\widehat{\rho(\gamma)}$ be the lift of $\rho(\gamma)$ to $\mathsf{SL}(d,\mathbb R)$ so that all of its eigenvalues 
are positive.
Suppose that $\{e_1(\rho(\gamma)),\ldots,e_d(\rho(\gamma))\}$ is a basis of $\mathbb R^d$ consisting
of eigenvectors for $\rho(\gamma)$ so that 
$$\widehat{\rho(\gamma)}(e_i(\gamma))=\lambda_i(\rho(\gamma)) e_i(\rho(\gamma))$$
for all $i$. 
Then we may write
$$\widehat{\rho(\gamma)}=\sum_{i=1}^d \lambda_i(\rho(\gamma))\p_i(\rho(\gamma))$$
where $\p_i(\rho(\gamma))$ is the projection onto the eigenline spanned by $e_i(\rho(\gamma))$
parallel to the  hyperplane spanned by the other basis elements.

In \cite{BCL}, we prove that if $\alpha$ and $\beta$ have non-intersecting axes 
and $\rho\in\Hn$, then the bases $\{e_i(\rho(\alpha))\}$ and $\{e_i(\rho(\beta))\}$ have strong
transversality properties, which generalize the transversality properties established by
Labourie in \cite{labourie-anosov}.

\begin{theorem}{\rm (\cite[Cor. 4.1]{BCL})}
\label{transversality}
If $\rho\in\Hn$, $\alpha,\beta\in\pi_1(S)-\{1\}$ and $\alpha$ and $\beta$ have non-intersecting
axes,  then any $d$ elements of
$$\{e_1(\rho(\alpha)),\ldots,e_{d}(\rho(\alpha)),e_{1}(\rho(\beta)),\ldots,e_d(\rho(\beta))\}$$
span $\mathbb R^d$. 
In particular, 
$$\p_i(\rho(\alpha))(e_j(\rho(\beta)))\ne 0$$
for any $i,j\in\{1,\ldots,d\}$.
\end{theorem}

If $\rho\in\Hn$ and ${\sf S}^2\rho:\pi_1(S)\to \sf{SL}({\sf S}^2(\mathbb R^d))$ is the second symmetric product of a lift of $\rho$ to
a representation into $\SL$, then
$${\sf S}^2\rho(\gamma)=
\sum_{i\le j}^{d} \lambda_i(\rho(\gamma))\lambda_j(\rho(\gamma))\p_{ij}(\rho(\gamma))$$
and if ${\sf E}^2\rho:\pi_1(S)\to \sf{SL}({\sf E}^2(\mathbb R^d))$ is the second exterior product of a lift of $\rho$ to a representation into $\SL$, then 
$${\sf E}^2\rho(\gamma)
=\sum_{i < j}^{d} \lambda_i(\rho(\gamma))\lambda_j(\rho(\gamma))\q_{ij}(\rho(\gamma))$$
where $\p_{ij}(\rho(\gamma))$  is  the projection onto the eigenline spanned by
$e_i(\rho(\gamma)) \cdot e_j(\rho(\gamma))$ and $\q_{ij}(\rho(\gamma))$ is the projection
onto the eigenline $e_i(\rho(\gamma))\wedge e_j(\rho(\gamma))$ parallel to the  hyperplane spanned by the other
products of basis elements. 
(Notice that ${\sf E}^2\rho$ and ${\sf S}^2\rho$ are independent of the choice of lift of $\rho$ to
a representation into $\SL$.)
Then
$$\p_{ii}(\rho(\gamma))(v\cdot  w) =\p_i(\rho(\gamma))(v)\cdot  \p_i(\rho(\gamma))(w),$$ 
$$\p_{ij}(\rho(\gamma))(v\cdot  w) = \p_i(\rho(\gamma))(v)\cdot  \p_j(\rho(\gamma))(w)+
\p_j(\rho(\gamma))(v)\cdot  \p_i(\rho(\gamma))(w)\quad \mbox{for } i\neq j,  \ {\rm and} $$
$$\q_{ij}(\rho(\gamma))(v\wedge w) = \p_i(\rho(\gamma))(v)\wedge \p_j(\rho(\gamma))(w)-
\p_j(\rho(\gamma))(v)\wedge \p_i(\rho(\gamma))(w).$$

We use Theorem \ref{transversality} to prove that various terms arising in our asymptotic analysis
are non-zero.

\begin{lemma}
\label{trace nonzero}
If $\alpha, \beta \in \pi_1(S)$ have non-intersecting axes and $\rho\in\Hn$, then
\begin{enumerate}
\item
$\Tr\big(\p_{ii}(\rho(\alpha))\p_{kk}(\rho(\beta))\big)\ne 0$, for all $i,k\in \{1,\ldots,d\}$,
\item
$\Tr\big(\q_{ij}(\rho(\alpha))\q_{kl}(\rho(\beta))\big)\ne 0$ if $i,j,k,l\in\{1,\ldots,d\}$, $i\ne j$ and $k\ne l$,
\item
$Tr\left(\p_{ii}(\rho(\alpha)){\sf S}^2\rho(\beta)\right)\ne 0$ if $i\in\{1,\ldots,d\}$, and
\item
$\Tr\left(\q_{ij}(\rho(\alpha)) {\sf E}^2\rho(\beta)\right)\ne 0$ if $i,j\in\{1,\ldots,d\}$ and $i\ne j$.
\end{enumerate}
\end{lemma}

\begin{proof}
We fix $\rho\in\Hn$ and identify $\rho(\gamma)$ with $\gamma$, for all $\gamma\in\pi_1(S)$,
throughout the proof for notational simplicity.
Choose bases $\{e_1(\alpha),\ldots,e_d(\alpha)\}$ and
$\{e_1(\beta),\ldots,e_d(\beta)\}$ and define
${\bf t}_{ij}(\alpha,\beta)$ so that
$$\p_i(\alpha)(e_j(\beta))={\bf t}_{ij}(\alpha,\beta) e_i(\alpha)$$
for all $i,j\in\{1,\ldots,d\}$.
Theorem \ref{transversality} implies that ${\bf t}_{ij}(\alpha,\beta)\ne 0$ for all $i$ and $j$, so
$$\Tr\big(\p_{ii}(\alpha)\p_{kk}(\beta)\big)={\bf t}_{ik}(\alpha,\beta)^2 {\bf t}_{ki}(\beta,\alpha)^2\ne 0.$$

If $i<j$ and $k<l$, we define 
$\s_{ijkl}(\alpha,\beta)$ by the equation 
$$e_k(\beta)\wedge e_l(\beta)\bigwedge_{r\ne i,j} e_r(\alpha)= \s_{ijkl}(\alpha,\beta)\left(
e_i(\alpha)\wedge e_j(\alpha)\bigwedge_{r\ne i,j} e_r(\alpha)\right).$$
Theorem \ref{transversality}  implies that $\s_{ijkl}(\alpha,\beta)\ne 0$, so
$$\Tr\big(\q_{ij}(\alpha)\q_{kl}(\beta)\big)=\s_{ijkl}(\alpha,\beta)\s_{klij}(\beta,\alpha)\ne 0.$$

Notice that we may choose the basis $\{e_i(\beta\alpha\beta^{-1})\}_{i=1}^d=\{\beta(e_i(\alpha))\}_{i=1}^d$,
in which case
$${\sf S}^2\rho(\beta)(e_i(\alpha)\cdot e_i(\alpha))=e_i(\beta\alpha\beta^{-1})\cdot e_i(\beta\alpha\beta^{-1}).$$
One then computes that
$$\Tr\left(\p_{ii}(\alpha){\sf S}^2\rho(\beta)\right)={\bf t}_{ii}(\alpha,\beta\alpha\beta^{-1})^2\ne 0.$$
(Notice that if $\alpha$ and $\beta$ have non-intersecting axes, then so do $\alpha$ and $\beta\alpha\beta^{-1}$.)

Similarly,
$$\Tr\left(\q_{ij}(\alpha){\sf  E}^2\rho(\beta)\right) = \s_{ijij}(\alpha,\beta\alpha\beta^{-1}) \ne 0.$$
\end{proof}

\subsection{Trace asymptotics}

If $\gamma\in\pi_1(S)$, let $\Lambda_{\alpha_1}^\gamma:\Hn\to \mathbb R$ be given by
$$\Lambda_{\alpha_1}^\gamma(\rho)=\frac{\lambda_1(\rho(\gamma))}{\lambda_2(\rho(\gamma))}$$
and notice that $L_{\alpha_1}^\gamma=\log \Lambda_{\alpha_1}^\gamma $. An asymptotic analysis of traces yields:

\begin{lemma}
\label{trace asymptotics}  
If $\alpha,\beta\in\pi_1(S)$ have non-intersecting axes and $\rho\in\Hn$, then
$$\lim_{n\rightarrow \infty} \frac{\Lambda_{\alpha_1}^{\alpha^n\beta^n}(\rho)}{\Lambda_{\alpha_1}^{\alpha^n}(\rho)\Lambda_{\alpha_1}^{\beta^n}(\rho)} =
\frac{\Tr(\p_{11}(\rho(\alpha))\p_{11}(\rho(\beta)))}{\Tr(\q_{12}(\rho(\alpha))\q_{12}(\rho(\beta)))} \neq 0$$
and
$$\lim_{n\rightarrow \infty} \frac{\Lambda_{\alpha_1}^{\alpha^n\beta}(\rho)}{\Lambda_{\alpha_1}^{\alpha^n}(\rho)} = 
\frac{\Tr\left(\p_{11}(\rho(\alpha)){\sf S}^2\rho(\beta)\right)}{\Tr\left(\q_{12}(\rho(\alpha){\sf E}^2\rho(\beta)\right)} \neq 0$$
\end{lemma}

\begin{proof}
We again fix $\rho\in\Hn$ and identify $\rho(\gamma)$ with $\gamma$ throughout the proof for notational simplicity.
One can compute that
$$ \frac{\Tr({\sf S}^2\rho(\alpha^n\beta^n))}{\Tr({\sf E}^2\rho(\alpha^n\beta^n))} = 
\frac{\sum_{1\le i\le j\le d} \lambda_i(\alpha^n\beta^n)\lambda_j(\alpha^n\beta^n)}
{\sum_{1\le i < j\le d} \lambda_i(\alpha^n\beta^n)\lambda_j(\alpha^n\beta^))} = 
\frac{\lambda_1(\alpha^n\beta^n)^2(1 + a_n)}{\lambda_1(\alpha^n\beta^n)\lambda_2(\alpha^n\beta^n)(1+b_n)}$$
Since $\lim_{n\to\infty}\frac{\lambda_j(\alpha^n\beta^n)}{\lambda_i(\alpha^n\beta^n)}= 0$ if $i>j$,
by Lemma \ref{hitchin expansion},
$a_n \rightarrow 0$ and  $b_n\rightarrow 0$ as $n\to\infty$.

Similarly, 
$$
 \frac{\Tr({\sf S}^2\rho(\alpha^n) {\sf S}^2\rho(\beta^n))}{\Tr({\sf E}^2\rho(\alpha^n) {\sf E}^2\rho(\beta^n))}  =
\frac{\Tr\left(\left(\sum_{i\le j} \lambda_i(\alpha^n)\lambda_j(\alpha^n)\p_{ij}(\alpha)\right)
\left(\sum_{ i\le j} \lambda_i(\beta^n)\lambda_j(\beta^n)\p_{ij}(\beta)\right)\right)}
{\Tr\left(\left(\sum_{i< j}\lambda_i(\alpha^n)\lambda_j(\alpha^n)\q_{ij}(\alpha)\right)
\left(\sum_{ i< j} \lambda_i(\beta^n)\lambda_j(\beta^n)\q_{ij}(\beta)\right)\right)}$$
$$=\Lambda_{\alpha_1}^{\alpha^n}(\rho)\Lambda_{\alpha_1}^{\beta^n}(\rho)
\frac{\Tr\big(\p_{11}(\alpha)\p_{11}(\beta)\big)( 1 + c_n)}{\Tr\big(\q_{12}(\alpha)\q_{12}(\beta)\big)( 1 + d_n)}$$
where $c_n \rightarrow 0$ and  $d_n\rightarrow 0$.
Since the two expression are equal, we may take limits to
obtain the first equality in the statement. Notice that  Lemma \ref{trace nonzero} is being used to guarantee
that $\Tr(\p_{11}(\alpha)\p_{11}(\beta))$ and $\Tr(\q_{12}(\alpha)\q_{12}(\beta))$ are non-zero
so that the right-hand expression makes sense and is non-zero.

To establish the second equality, we compute that
$$ \frac{\Tr({\sf S}^2\rho(\alpha^n\beta))}{\Tr({\sf E}^2\rho(\alpha^n\beta))}=
 \frac{\lambda_1(\alpha^n\beta)^2(1 + a'_n)}{\lambda_1(\alpha^n\beta)\lambda_2(\alpha^n\beta)(1+b'_n)},$$
where $a_n' \rightarrow 0$ and  $b_n'\rightarrow 0$,
and that
$$\frac{\Tr({\sf S}^2\rho(\alpha^n){\sf S}^2\rho(\beta))}{\Tr({\sf E}^2\rho(\alpha^n){\sf E}^2\rho(\beta))} = 
\frac{\Tr\left(\left(\sum_{i\le j} \lambda_i(\alpha^n)\lambda_j(\alpha^n)\p_{ij}(\alpha)\right)
{\sf S}^2\rho(\beta)\right)}
{\Tr\left(\left(\sum_{i< j} \lambda_i(\alpha^n)\lambda_j(\alpha^n)\q_{ij}(\alpha)\right){\sf E}^2\rho(\beta)\right)}$$
$$ = \Lambda_{\alpha_1}^{\alpha^n}(\rho)\frac{\Tr\big(\p_{11}(\alpha){\sf S}^2(\rho(\beta)\big)( 1 + c'_n)}
{\Tr\big(\q_{12}(\alpha){\sf E}^2\rho(\beta)\big)( 1 + d'_n)}$$
where $c_n' \rightarrow 0$ and  $d_n'\rightarrow 0$.
We obtain the second equation by 
setting the two expressions above equal, taking limits and applying Lemma \ref{trace nonzero} to guarantee
that the right-hand expression makes sense and is non-zero.
\end{proof}

\subsection{Derivatives of eigenvalue functions}
Let $\lambda_i^\gamma:\Hn\to \mathbb R$
be given by $\lambda_i^\gamma(\rho)=\lambda_i(\rho(\gamma))$.

\begin{proposition}
\label{type zero translation}
If $v\in T_\rho\mathcal H_d(S)$ and ${\rm D}_\rho L_{\alpha_1}^\gamma(v)=0$ for all $\gamma\in\pi_1(S)$,
then ${\rm D}_\rho\lambda_i^\gamma(v)=0$, for all $i=1,\ldots,d$ and all $\gamma\in\pi_1(S)$.
\end{proposition}

Notice that the assumptions of Proposition \ref{type zero translation} are equivalent to the assumption
that \hbox{${\rm D}_\rho \Lambda_{\alpha_1}^\gamma(v)=0$} for all $\gamma\in\pi_1(S)$.
The proof of Proposition \ref{type zero translation} makes use of the following elementary lemma:

\begin{lemma}
\label{technical}
Let $a_i, b_i, c_i, d_i, w_i \in \Real$, for $i = 1,\ldots k$, with $w_1 > w_2 > \ldots > w_k >0$. 
If, for every $n \in \mathbb  N$,
$$\sum_{i=1}^k (a_i+nb_i)w_i^n = \sum_{i=1}^k (c_i+n d_i)w_i^n ,$$
then $a_i = c_i$ and $b_i= d_i$ for all $i$.
\end{lemma}

\begin{proof}
We first divide by $nw_1^n$ and take the limit to see that
$$b_1 = \lim_{n\rightarrow \infty} \frac{1}{nw_1^n} \left(\sum_{i=1}^k (a_i+nb_i)w_i^n\right) =  \lim_{n\rightarrow \infty} \frac{1}{nw_1^n} \left(\sum_{i=1}^k (c_i+nd_i)w_i^n\right) = d_1.$$
We then subtract $nb_1w_1^n$ from each side, divide by $w_1^n$, 
and pass to a limit to conclude that $a_1 = c_1$. 

We may then remove the first order terms and proceed iteratively.
\end{proof}

\medskip\noindent
{\em Proof of Proposition \ref{type zero translation}.}
We will show that, if $\gamma\in\pi_1(S)$, then
${\rm D}_\rho(\log\lambda_i^\gamma)(v)= {\rm D}_\rho(\log\lambda_1^\gamma)(v)$ for all $i$.
Since $\lambda_1^\gamma\cdots\lambda_d^\gamma = 1$, 
$${\rm D}_\rho(0)(v)={\rm D}_\rho(\log\lambda_1^\gamma)(v) +\cdots+ {\rm D}_\rho(\log\lambda_d^\gamma)(v) =
d\ {\rm D}_\rho(\log\lambda_1^\gamma)(v) = 0,$$
which in turn implies that  ${\rm D}_\rho\lambda_i^\gamma(v) = 0$ for all $i$. 

We first notice that, since ${\rm D}_\rho L_{\alpha_1}^\gamma(v)=0$,
${\rm D}_\rho(\log\lambda_2^\gamma)(v)=  {\rm D}_\rho(\log\lambda^\gamma_1)(v)$ for all $\gamma\in\pi_1(S)$.
We proceed iteratively.
Assume that ${\rm D}_\rho(\log\lambda_i^\gamma)(v)=  {\rm D}_\rho(\log\lambda^\gamma_1)(v)$ for all $i < m$ and $\gamma\in\pi_1(S)$. 
Notice that this is equivalent to the claim that ${\rm D}_\rho\lambda_i^\gamma(v)=  {\rm D}_\rho\lambda_1^\gamma(v)$
for all $i < m$ and $\gamma\in\pi_1(S)$. 

Fix $\alpha\in\pi_1(S)-\{1\}$ and
let $\beta$ be an element of $\pi_1(S)$, so that $\alpha$ and $\beta$ have non-intersecting axes
and consider the family of analytic functions $\{ F_n:\Hn\to \mathbb R\}_{n\in\mathbb N}$ defined by
$$F_n(\rho)= \frac{\left(\frac{\Tr(\p_{11}(\rho(\alpha)){\sf S}^2\rho(\beta^n))}{\Tr(\q_{12}(\rho(\alpha)){\sf E}^2\rho(\beta^n))}\right)}
{(\Lambda_1^\beta(\rho))^n\left(\frac{\Tr\big(\p_{11}(\rho(\alpha))\p_{11}(\rho(\beta))\big)}{\Tr\big(\q_{12}(\rho(\alpha))\q_{12}(\rho(\beta))\big)}\right)}.$$
Notice that, by Lemma \ref{trace asymptotics}, the numerator of $F_n$ is an analytic function which is a limit of 
analytic functions which, by assumption, have derivative zero in the direction $v$, so the numerator has derivative zero
in direction $v$. We may similarly use our assumptions and
Lemma \ref{trace asymptotics} to show that the denominator of $F_n$ has derivative zero in direction $v$.
Therefore, ${\rm D}_\rho F_n(v)=0$ for all $n\in\mathbb N$.

We adopt the shorthand $\lambda_i = \lambda_i^\rho(\beta)$ and expand the above equation to
see that
$$F_n(\rho) = \frac{\sum_{i \le j} a_{ij}(\rho)\left(\frac{\lambda_i}{\lambda_1}\right)^n \left(\frac{\lambda_j}{\lambda_1}\right)^n}{\sum_{i < j} b_{ij}(\rho)\left(\frac{\lambda_i}{\lambda_1}\right)^n \left(\frac{\lambda_j}{\lambda_2}\right)^n} = \frac{\sum_{i \le j} a_{ij}(\rho)u_i^nu_j^n}{\sum_{i < j} b_{ij}(\rho)u_i^nv_j^n}$$
where 
$$a_{ij}(\rho)= \frac{\Tr(\p_{11}(\rho(\alpha))\p_{ij}(\rho(\beta)))}{\Tr(\p_{11}(\rho(\alpha))\p_{11}(\rho(\beta)))}, \ \
b_{ij}(\rho)= \frac{\Tr(\q_{12}(\rho(\alpha))\q_{ij}(\rho(\beta)))}{\Tr(\q_{12}(\rho(\alpha))\q_{12}(\rho(\beta)))} ,
\ \ u_i = \frac{\lambda_i}{\lambda_1}, \ {\rm and} \ v_i = \frac{\lambda_i}{\lambda_2}.$$
In particular, $a_{11} = b_{12} = 1$ and, by Lemma  \ref{trace nonzero}, 
$b_{1m} \neq 0$ for all $m$. Since ${\rm D}_\rho F_n(v)=0$ for all $n\in\mathbb N$,
\begin{equation*}
\left({\rm D}_\rho\left(\sum_{i\le j} a_{ij} u_i^n u_j^n\right)(v) \right)\ \left(\left(\sum_{k< l} b_{kl} u_k^n v_l^n\right)(\rho)\right)=   \left(\left(\sum_{i\le j} a_{ij} u_i^n u_j^n\right)(\rho) \right) \ \left({\rm D}_\rho\left(\sum_{k< l} b_{kl} u_k^n v_l^n\right)(v)\right)
\label{eq1}
\end{equation*}
Letting $U_{ijkl} = u_iu_ju_kv_l$, this becomes
$$\sum_{i\le j, k<l} \left(\dot{a}_{ij}b_{kl}+ n a_{ij} b_{kl}\left(\frac{\dot{u}_i}{u_i} + \frac{\dot{u}_j}{u_j}\right)\right)U_{ijkl}^n 
= \sum_{i\le j, k<l} \left(a_{ij}\dot{b}_{kl}+ n a_{ij} b_{kl}\left(\frac{\dot{u}_k}{u_k} + \frac{\dot{v}_l}{v_l}\right)\right)U_{ijkl}^n.$$

We group terms where $U_{ijkl}$ agree and order so that, as sets, $\{w_s\}_{s=1}^M = \{U_{ijkl}\}_{i\le j, k<l}$
and $w_i > w_{i+1}>0$ for all $i$. We may rewrite the expression above as
$$\sum_{s=1}^M (A_s+nB_s)w_s^n= \sum_{s=1}^M (C_s+nD_s)w_s^n$$
where $A_s$, $B_s$, $C_s$ and $D_s$ are constants depending only on $s$ and
not on $n$. Lemma \ref{technical} implies that $A_s= C_s$ and $B_s = D_s$ for all $s$.

By our iterative hypothesis, ${\rm D}_\rho(\log\lambda_i)(v)=  {\rm D}_\rho(\log\lambda_1)$ for all $i < m$, and $m > 2$. 
Therefore,  $\dot{u}_i = \dot{v}_i=0$
for all $i < m$. Since $m>2$, the iterative step of the proof will be completed if 
we show that either $\dot{u}_m = 0$ or $\dot{v}_m = 0$.

Consider $s_1$ such that $w_{s_1} = U_{111m} = v_m$ and notice that
$$B_{s_1} = \sum_{\{i\le j,\ k<l\ |\ U_{ijkl} = v_m\}}\left(a_{ij} b_{kl}\left(\frac{\dot{u}_i}{u_i} + \frac{\dot{u}_j}{u_j}\right)\right) .$$
If $w_{s_1} = U_{ijkl}$, then $U_{ijkl} =u_iu_ju_kv_l = v_m$. Since
$1=u_1>u_2>\cdots>u_d>0$ and $1\ge v_i>u_i$ for all $i\ge 2$, we see that
$u_j  \ge u_iu_ju_kv_l =v_m > u_m$, so $i\le j < m$.  Since $\dot{u}_i=0$ if $i<m$,
we see that $B_{s_1}=0$.

A similar analysis yields that if  $U_{ijkl} =u_iu_ju_kv_l = v_m$, then $k<m$ and $l \le m$. Therefore, $\dot u_k=0$
and $\dot v_l=0$ if $l\ne m$. However, if $l=m$, then $i=j=k=1$, so
$$D_{s_1} =  \sum_{\{i\le j\ ,k<l \ |\ U_{ijkl} = v_m\}}\left(a_{ij} b_{kl}\left(\frac{\dot{u}_k}{u_k} + \frac{\dot{v}_l}{v_l}\right)\right)
=  a_{11}b_{1m}\left( \frac{\dot{v}_m}{v_m}\right) = b_{1m}\left( \frac{\dot{v}_m}{v_m}\right),$$
so, since $D_{s_1} = B_{s_1} = 0$,  we conclude that $b_{1m}\dot{v}_m = 0$. Since we have previously observed
that $b_{1m}\ne0$, it must be that $\dot{v}_m=0$ which completes the proof.
\eproof

\subsection{Proof of Theorem \ref{cotangent intro}.}
If $v\in T_\rho\Hn$ and ${\rm D}_\rho L_{\alpha_1}^\gamma(v)=0$ for all $\gamma\in\pi_1(S)$, then, by
Proposition \ref{type zero translation}, ${\rm D}_\rho\lambda_i^\gamma(v)=0$ for all $i$ and all $\gamma\in\pi_1(S)$.
However, Proposition 10.1 in \cite{BCLS} guarantees that
$\{ {\rm D}_\rho\lambda_1^\gamma\}_{\gamma\in \pi_1(S)}$ generates the cotangent space to $\Hn$ at $\rho$, so our proof is
complete.
\eproof

\section{Degeneracy of $\PP_{\alpha_n}$ on $\mathcal H_{2n}(S)$}

Bridgeman \cite{bridgeman-wp} showed that the pressure metric on quasifuchsian space is
degenerate on the Fuchsian locus.
In \cite[Section 5.8]{BCS-survey}, we construct a pressure metric on $\Hn$
which is associated to the Hilbert length of elements of the image and similarly prove
that this metric is  degenerate on the fixed point locus of the contragredient involution. 
A very similar argument yields that $\PP_{\alpha_n}$ is degenerate on $\mathcal H_{2n}(S)$.

Recall that the contragredient involution $\tau:\mathcal H_{2n}(S)\to \mathcal H_{2n}(S)$  fixes the submanifold
$\mathcal H(S,\mathsf{PSp}(2n))$ of Hitchin representations with image in $\mathsf{PSp}(2n)$.

\begin{proposition}
\label{Pn degenerate on H2n}
The pressure quadratic form $\PP_{\alpha_n}$ on $\mathcal H_{2n}(S)$ is degenerate on $\mathcal H(S,\mathsf{PSp}(2n))$.
In particular,  if \hbox{$\rho\in\mathcal H(S,\mathsf{PSp}(2n))$}, $v\in {\sf T}_\rho\Hn$ and $D\tau_\rho( v) = - v$, 
then  $||v||_{\PP_{\alpha_n}} = 0$. 
\end{proposition}

\begin{proof}
Suppose that $\rho\in\mathcal H(S,\mathsf{PSp}(2n))$, $v\in T_\rho\Hn$ and $D\tau_\rho(v) = -v$.
We choose a path $\{\rho_t\}_{t\in (\epsilon,\epsilon)}$ in $\mathcal H_{2n}(S)$ such that $\dot \rho_0 = v$ and $\tau (\rho_t) = \rho_{-t}$ 
for all $t\in (-\epsilon,\epsilon)$. Since $\lambda_i(\sigma(\gamma^{-1}))=\left(\lambda_{2n-i}(\tau(\sigma)(\gamma))\right)^{-1}$
for all $i$ and all $\sigma\in\mathcal H_{2n}$,
$$L_{\alpha_n}(\rho_t(\gamma)) = \log\left(\frac{\lambda_{n}(\rho_t(\gamma))}{\lambda_{n+1}(\rho_t(\gamma))}\right)
= L_{\alpha_n}(\rho_{-t}(\gamma))$$
for all $t\in (-\epsilon,\epsilon)$ and $\gamma\in\pi_1(S)$.
Therefore, 
$$\frac{d}{dt}\Big|_{t=0} L_{\alpha_n}(\rho_t(\gamma)) = 0$$
for all $\gamma\in\pi_1(S)$, so Proposition \ref{pressure forms} implies that
$\|\dot\rho_0\|_{\PP_{\alpha_n}}=\| v\|_{\PP_{\alpha_n}} = 0$.
\end{proof}

It is natural to wonder whether a similar symmetry is responsible for all degeneracies
of pressure metrics constructed in this fashion.

\end{document}